\documentclass{article}
\usepackage{bm}
\usepackage{tikz,mathpazo}
\usepackage{pgf}
\usepackage[top=2.5cm, bottom=2.5cm, left=3cm, right=3cm]{geometry}   
\usepackage[normalem]{ulem}
\usepackage{enumitem}
\usepackage{indentfirst}
\usepackage{graphicx}
\usepackage{graphics}
\usepackage[toc,page,title,titletoc,header]{appendix}
\usepackage{bm}
\usepackage{bbm}
\usepackage{listings}  
\usepackage{amsmath}
\usepackage{setspace} 
\usepackage{indentfirst}
\usepackage{caption}
\usepackage{multirow} 
\usepackage{lipsum,multicol}
\usepackage{pdfpages}
\usepackage{float}
\usepackage{amsthm}
\usepackage{pdfpages}
\usepackage{url}
\usepackage{colortbl}
\usepackage{subfigure}
\usepackage{epsfig}
\usepackage{epstopdf}
\usetikzlibrary{shapes.geometric, arrows}
\usepackage{fancyhdr}
\usepackage{abstract}
\usepackage{tikz,mathpazo}
\usetikzlibrary{shapes.geometric, arrows}
\usepackage{amssymb}
\usepackage{latexsym}
\usepackage{verbatim}
\usepackage[numbers]{natbib} 
\usepackage{booktabs}
\usepackage{mathrsfs}

\usepackage{hyperref}
\hypersetup{
    colorlinks=true,
    linkcolor=blue,
    filecolor=magenta,      
    urlcolor=cyan,
    citecolor = cyan,
    }

\newcommand{\bb}[1]{\mathbb{#1}}

\newcommand{\ca}[1]{\mathcal{#1}}

\newcommand{\xk}{{x_{k} }}
\newcommand{\yk}{{y_{k} }}

\newcommand{\zk}{{z_{k} }}

\newcommand{\zkp}{{z_{k+1} }}

\newcommand{\Hb}{{\bar H}}
\newcommand{\Ht}{{\tilde H}}
\newcommand{\Hbk}{{\bar H}_k}

\newcommand{\xik}{\xi_k}
\newcommand{\muy}{\mu}

\newcommand{\Rn}{\mathbb{R}^n}

\newcommand{\Rm}{\mathbb{R}^m}

\DeclareMathOperator*{\argmin}{arg\,min}

\newcommand{\ys}{\mathcal{Y}^*}
\newcommand{\zs}{x, \ys(x)}
\newcommand{\crit}{\mathrm{crit}}
\newcommand{\cB}{\ca{B}}

\newcommand{\dfy}{\nabla_y f}
\newcommand{\ddfxx}{\nabla^2_{xx} f}
\newcommand{\ddfyy}{\nabla^2_{yy} f}
\newcommand{\ddfxy}{\nabla^2_{xy} f}
\newcommand{\ddfyx}{\nabla^2_{yx} f}

\newcommand{\qaq}{\quad\text{and}\quad}
\newcommand{\hb}{h_\beta}

\newcommand{\df}{\nabla f}
\newcommand{\ddf}{\nabla^2 f}
\newcommand{\dhb}{\nabla h_{\beta}}

\newcommand{\cF}{\mathcal{F}}
\newcommand{\Exp}{\mathbb{E}}

\newcommand{\cBgky}{\mathcal{B}^g_{k,y}}
\newcommand{\cBhko}{\mathcal{B}^{h}_{k,1}}
\newcommand{\cBhkt}{\mathcal{B}^{h}_{k,2}}

\newcommand{\cBgko}{\mathcal{B}^{g}_{k,1}}
\newcommand{\cBgkt}{\mathcal{B}^{g}_{k,2}}

\newcommand{\buj}{\boldsymbol{u}_j}

\newcommand{\bzj}{\boldsymbol{z}_j}
\newcommand{\bzjm}{\boldsymbol{z}_{j-1}}

\newcommand{\bgkm}{\boldsymbol{g}_{k-1}}
\newcommand{\bgk}{\boldsymbol{g}_k}
\newcommand{\bHbk}{\bar{\boldsymbol{H}}_k}
\newcommand{\bHbkm}{\bar{\boldsymbol{H}}_{k-1}}
\newcommand{\bzk}{\boldsymbol{z}_k}
\newcommand{\bzkm}{\boldsymbol{z}_{k-1}}
\newcommand{\bzkp}{\boldsymbol{z}_{k+1}}
\newcommand{\bgyk}{\boldsymbol{g}_{y,k}}
\newcommand{\bxik}{\boldsymbol{\xi}_k}
\newcommand{\bxikm}{\boldsymbol{\xi}_{k-1}}
\newcommand{\ddF}{\nabla^2 F} 
\newcommand{\dF}{\nabla F}

\newcommand{\alo}{\alpha_1}
\newcommand{\alt}{\alpha_2}
\newcommand{\sigo}{\sigma_1}
\newcommand{\sigt}{\sigma_2}
\newcommand{\bwk}{\boldsymbol{w}_k}
\newcommand{\bwkm}{\boldsymbol{w}_{k-1}}
\newcommand{\cL}{\mathcal{L}}
\newcommand{\bxis}{\boldsymbol{\xi}_{k^*}}
\newcommand{\bzs}{\boldsymbol{z}_{k^*}}
\newcommand{\bzsp}{\boldsymbol{z}_{k^*+1}}
\newcommand{\xmark}{\color{red}\ding{55}}
\newcommand{\cmark}{\color{black}\ding{51}}
\newcommand{\Ek}{\Exp_k}

\newtheorem{theo}{Theorem}[section]
\newtheorem{lem}[theo]{Lemma}
\newtheorem{prop}[theo]{Proposition}

\newtheorem{coro}[theo]{Corollary}

\newtheorem{defin}[theo]{Definition}
\newtheorem{assumpt}[theo]{Assumption}

\theoremstyle{definition}

\usepackage[figuresright]{rotating}
\usepackage{pdflscape}

\usepackage{caption}
\usepackage{algorithm}
\usepackage{algpseudocode}
\usepackage{longtable}
\usepackage{appendix}
\usepackage{authblk}
\usepackage{pifont}

\numberwithin{equation}{section}

\usepackage{array}

\title{A Single-Loop Regularized Newton Method for Nonconvex-Strongly-Concave Minimax Optimization}
\author{Bohao Ma\thanks{School of Data Science, The Chinese University of Hong Kong (Shenzhen), Shenzhen, Guangdong, China 
(\url{bohaoma@link.cuhk.edu.cn}  and \url{xncxy@cuhk.edu.cn}).}, ~
Nachuan Xiao\protect\footnotemark[1], ~
Junyu Zhang\thanks{Department of Industrial Systems Engineering and Management, National University of Singapore, Singapore (\url{junyuz@nus.edu.sg}).}}

\begin{document}
\maketitle
	
\begin{abstract}
    For smooth nonconvex-strongly-concave minimax problems, existing second-order methods share a common double-loop structure where the inner maximization is solved to sufficiently high accuracy before each second-order step. First-order methods are often adopted in the inner loops for scalability, but they also undermine the condition-insensitivity of second-order methods, limiting these methods to instances with mild conditioning. To resolve this issue, we propose a novel single-loop framework based on an equivalent  regularized minimization reformulation of the original problem.  By deriving a new adaptive cubic-quadratic majorization to dynamically absorb the non-Lipschitz components of the reformulated Hessian, we establish a regularized Newton method with robust theoretical guarantees across multiple settings. For deterministic problems, our single-loop method matches the $\mathcal{O}(\varepsilon^{-1.5})$ global iteration complexity of the existing double-loop second-order methods, while automatically achieving a local superlinear rate that is unavailable in existing works due to the inner-loop bottleneck. For the stochastic setting, we achieve  $\mathcal{O}(\varepsilon^{-3})$ gradient and $\mathcal{O}(\varepsilon^{-2})$ Hessian complexities by integrating a recursive variance reduction, strictly improving  
    those of the double-loop methods by  $\mathcal{O}(\varepsilon^{-0.5})$ factors. In both deterministic and stochastic experiments, our methods significantly outperform the benchmarks, offering substantial speedups over the double-loop methods even under mildly conditioned instances. 
    As a byproduct of our analysis, we close a gap in stochastic second-order methods for nonconvex minimization, where the best known result contains a nontrivial technical issue.  
\end{abstract}
	
\section{Introduction} \label{sec:intro}
In this paper, we consider the nonconvex-strongly-concave (NC-SC) minimax problem 
\begin{equation} 
    \label{Prob_Ori}
    \tag{MM}
    \begin{aligned}
        \min_{x \in \Rn} \max_{y \in \Rm} f(x,y), 
    \end{aligned}
\end{equation}
where the objective function $f: \Rn \times \Rm \to \bb{R}$ is smooth, possibly nonconvex in $x$, and strongly concave in $y$. 
Besides the classic applications in generative models \citep{goodfellow2020generative}, adversarial learning \citep{sinha2017certifying}, and distributionally robust regression \citep{shafieezadeh2015distributionally}, efficiently solving problem \eqref{Prob_Ori} has also become a key step in the recent development of bilevel optimization methods \citep{kwon2023fully,chen2025near}.  
In many of these applications, the problem exhibits a stochastic structure, where the objective function takes the expectation form $f(x,y) = \Exp_\zeta[F(x,y;\zeta)]$.

While first-order methods \citep[e.g.,][]{lin2020gradient,mahdavinia2022tight,li2023tiada,xu2023unified,lin2025two} have become methods of choice for large-scale problems, they often struggle with ill-conditioning issues and converge slowly in complex landscapes. In contrast, for problems of small- to medium-scale, second-order methods exhibit substantial advantages. Leveraging curvature information, they not only enjoy faster local convergence and better dependence on the condition number, but also provide principled mechanisms for escaping strictly degenerate first-order stationary points, see, e.g., \citep{nesterov2006cubic,royer2020newton}. Motivated by these benefits, there has long been interest in designing efficient second-order methods for minimax problems \citep{taji1993globally}, dating back to the last century. However, their recent development has been much more limited than that of first-order methods.

In the design and analysis of optimization methods for solving \eqref{Prob_Ori}, one key issue lies in the lack of a suitable descent Lyapunov function (or potential function), which has posed challenges for both algorithm design and convergence analysis. In particular, for NC-SC problems considered in this paper, existing second-order methods are mostly based on the following minimization reformulation: 
\begin{equation} 
\tag{VM}
\label{Prob_value_fcn}
    \min_{x \in \Rn} \Phi(x), \quad \text{with} \quad \Phi(x):= \max_{y \in \Rm} f(x,y),
\end{equation}
where $\Phi(x)$ is often called the value function and hence we call it the value minimization \eqref{Prob_value_fcn}. Due to the inner maximization in \eqref{Prob_value_fcn}, these methods  usually adopt a double-loop structure \citep[etc.]{luo2022finding,chen2023a,wang2025gradient,chen2026homogeneous,yang2026second}, which suffers from several significant drawbacks.       

For deterministic NC-SC problems, the inner loop directly undermines two traditional advantages of second-order methods: \emph{local superlinear convergence} and \emph{reduced sensitivity to ill-conditioning}. Typical examples include \citep{luo2022finding} and \citep{chen2026homogeneous}, whose  outer loops adopt Nesterov's cubic regularized Newton \citep{nesterov2006cubic} and Ye's homogeneous second-order descent \citep{zhang2025homogeneous} frameworks, respectively. However, both methods use first-order methods to solve the inner-loop maximization subproblem to \emph{sufficiently high accuracy}. Although this design improves the scalability of the algorithms by adopting cheaper inner-loop subroutines, it also sacrifices local superlinear convergence and reintroduces condition-number dependence. For ill-conditioned problems, these issues significantly limit the practical performance of double-loop methods (see Section \ref{sec:numerical experiments}).

For stochastic NC-SC problems, the inner loop further worsens the Hessian and gradient sample complexities. On the one hand,  double-loop second-order methods usually require the precision of the inner loop to be much higher than the target precision of the outer loop. On the other hand, since stochastic first-order methods no longer enjoy linear convergence for the inner problems, the $\mathrm{poly}(1/\varepsilon)$ and $\mathrm{poly}(\kappa)$ terms\footnote{$\varepsilon$ denotes the target accuracy and $\kappa$ denotes the condition number of the strongly concave inner problem. }  that were once suppressed logarithmically by linearly convergent deterministic first-order methods  become explicit factors, causing significantly worse complexity dependence on $\varepsilon$ and $\kappa$. 

The above discussion of the limitations of double-loop second-order methods motivates the main question of our research: 
\begin{quote}
    \emph{Can we design a simple single-loop second-order method for NC-SC problems to address the aforementioned issues?} 
\end{quote}
In this paper, we provide an affirmative answer to this question by considering the following smooth Regularized Minimization \eqref{Prob_Min} reformulation of \eqref{Prob_Ori}: 
\begin{equation}
\tag{RM}\label{Prob_Min}
\min_{x\in\Rn, y\in\Rm} \quad \hb(x,y)
:= f(x,y) + \frac{\beta}{2} \|\nabla_y f(x,y)\|^2.
\end{equation}
Then we use $h_\beta$ as the Lyapunov function to drive convergence. 
In a previous work \citep{ma2026linesearch}, we established that problems \eqref{Prob_Min} and \eqref{Prob_Ori} share the same set of first- and second-order stationary points as well as global and local optimal points.
Most importantly, unlike the value function $\Phi$, the evaluation of $\hb$ does not rely on any inner maximization subroutine, paving the way for developing single-loop second-order methods. This removes the inner-loop obstacle, but does not yet allow a direct application of standard second-order methods. 

Indeed, naively applying the standard second-order methods \citep[etc.]{nesterov2006cubic,zhang2025homogeneous} to the regularized minimization problem requires computing the third-order derivatives of $f$, and it leads to theoretical difficulties as these methods typically require the objective Hessian to be Lipschitz continuous, whereas $\nabla^2 \hb$ is not necessarily so. In fact, $\nabla^2 \hb$ is impossible to be Lipschitz continuous in general, as it requires the zeroth-, first-, second-, and third-order Lipschitz continuity of $f$, while the zeroth-order Lipschitz condition is in direct conflict with the strong concavity in the $y$ variable. 

To resolve this bottleneck, we propose a novel adaptive cubic-quadratic (ACQ) majorization that only requires the standard first- and second-order Lipschitz conditions on $f$. Leveraging the special structure of $h_\beta$, the ACQ scheme absorbs the non-Lipschitz components of $\nabla^2h_\beta$ (that involves $\nabla^3f$) into a carefully designed adaptive regularizer, which provably diminishes as the algorithm converges. This strategy allows us to construct a tightly bounded, computationally tractable majorizer without invoking additional third-order smoothness assumptions and third-order derivative computation. 

Building upon the ACQ majorization, our contributions are summarized as follows:\vspace{0.1cm}
\begin{enumerate}[leftmargin=0.5cm, itemsep=1pt, topsep=0pt, parsep=0pt]
    \item \textbf{A single-loop second-order method with local superlinear rate.} For the deterministic NC-SC minimax setting, we introduce Algorithm \ref{alg:cubic-minimax}, a single-loop second-order method that achieves a  condition-number-free local superlinear rate under local nondegeneracy conditions, while matching the state-of-the-art global iteration complexity results (see Table \ref{tab:minimax_complexity}).  
    \item \textbf{A single-loop stochastic second-order method with improved $\varepsilon$-dependence.} For the stochastic NC-SC minimax setting, we introduce Algorithm \ref{alg:stoch-cubic-minimax}, a single-loop stochastic second-order method that avoids the stochastic first-order subroutines in the inner loops required by the previous work \citep{chen2023a}. Utilizing a recursive variance reduction alongside independent sampling techniques, we improve both the gradient and Hessian sample complexities by a factor of $\mathcal{O}(\varepsilon^{-0.5})$. See Table \ref{tab:minimax_complexity}.
    \item \textbf{Resolving the theoretical gap in stochastic nonconvex minimization.} We observe that the state-of-the-art Hessian and gradient sample complexity result \citep{zhou2020stochastic} for stochastic minimization has a nontrivial technical flaw that stems from the structural conflict between a conditional high-probability truncation argument and the martingale-difference argument\footnote{Specifically, in the proofs of Lemmas A.3 and A.4, \citep{zhou2020stochastic} erroneously substitute high-probability bounds directly into Azuma-type vector and matrix concentration inequalities. Conditioning on the events where these bounds hold inherently destroys the underlying martingale difference structure required by the concentration inequalities. This structural conflict cannot be fixed easily. Our paper, in the special case when setting $\beta=0$, can be viewed as an alternative in-expectation analysis via the Burkholder-Davis-Gundy (BDG) inequality.}. The same technical issue repeatedly appears in many of its follow-ups. As an important byproduct of our paper, we specifically present the results of Algorithm \ref{alg:stoch-cubic-minimax} in such a way that, when setting $\beta = 0$, the analysis and sample complexity bounds of the minimax setting automatically apply to the minimization setting, recovering the state-of-the-art sample complexities with rigorous proof. See  Table \ref{tab:minimization_complexity}.
\end{enumerate}

\begin{table}[htbp]
    \centering
    \begin{tabular}{c|c|c|c|c}
        \toprule
        \multicolumn{5}{c}{\textbf{Deterministic Setting}} \\
        \midrule
        \textbf{Reference} & \textbf{Iter.} & \multicolumn{2}{c|}{\textbf{Local superlinear convergence}} & \textbf{Stat.} \\
        \midrule
        \cite{luo2022finding} 
        & $\mathcal{O}(\kappa^{1.5}\sqrt{\rho}\varepsilon^{-1.5})$ & \multicolumn{2}{c|}{\xmark} & $\Phi$ \\
        \cite{wang2025gradient} & $\tilde{\mathcal{O}}(\kappa^{1.5}\sqrt{\rho}\varepsilon^{-1.5})$ & \multicolumn{2}{c|}{\xmark} & $\Phi$ \\
        \cite{chen2026homogeneous} & $\tilde{\mathcal{O}}(\kappa^{1.5}\sqrt{\rho}\varepsilon^{-1.5})$ & \multicolumn{2}{c|}{\xmark} & $\Phi$ \\
        This work (Alg{.} \ref{alg:cubic-minimax}) & $\mathcal{O}(\sqrt{\kappa\rho}\varepsilon^{-1.5})$ & \multicolumn{2}{c|}{\cmark} & $\hb$ \\
        \midrule
        \midrule
        \multicolumn{5}{c}{\textbf{Stochastic Setting}} \\
        \midrule
        \textbf{Reference} & \textbf{Iter.} & \textbf{Grad.} & \textbf{Hess.} & \textbf{Stat.} \\
        \midrule
        \cite{chen2023a} & $\mathcal{O}(\kappa^{1.5}\sqrt{\rho}\varepsilon^{-1.5})$ & $\tilde{\mathcal{O}}(\kappa^{3.5}\sqrt{\rho}\varepsilon^{-3.5})$ & $\tilde{\mathcal{O}}(\kappa^{2.5}L^2\rho^{-0.5}\varepsilon^{-2.5})$ & $\Phi$ \\
        This work (Alg{.} \ref{alg:stoch-cubic-minimax}) & $\mathcal{O}(\sqrt{\kappa\rho}\varepsilon^{-1.5})$ & $\mathcal{O}([L^{1.5}\kappa^{0.5}+\rho]\kappa\varepsilon^{-3})$ & $\mathcal{O}(\kappa L\varepsilon^{-2})$ & $\hb$ \\
        \bottomrule
    \end{tabular}
    \vspace{0.1cm} 
    \caption{Summary of complexity results of the second-order methods in the NC-SC minimax setting. ``Iter.,'' ``Grad.,'' ``Hess.,'' and ``Stat.'' stand for iteration complexity, (stochastic) gradient sample complexity, (stochastic) Hessian sample complexity, and stationarity measure, respectively. $L$ and $\rho$ represent the Lipschitz moduli of $\df$ and $\ddf$, respectively. $\kappa := L/\muy$ denotes the condition number, where $\mu$ is the strong concavity modulus of $f(x,\cdot)$. The double-loop methods in \citep{luo2022finding,chen2023a} find an ($\varepsilon$, $\sqrt{\varepsilon}$)-SOSP of $\Phi$, whereas our single-loop methods find an ($\varepsilon$, $\sqrt{\varepsilon}$)-SOSP of $\hb$. Although an $(\varepsilon, \sqrt{\varepsilon})$-SOSP of $\hb$ corresponds to a $(\kappa\varepsilon, \kappa\sqrt{\varepsilon})$-SOSP of $\Phi$, we emphasize that this transformation only represents a pessimistic worst-case theoretical guarantee.
    \endgraf 
    }
    \label{tab:minimax_complexity}
    \vspace{0.5cm}
    \begin{tabular}{c|c|c|c|c}
        \toprule
        \textbf{Reference} & \textbf{Setting} & \textbf{Iter. } & \textbf{Grad.} & \textbf{Hess.} \\
        \midrule
        \cite{nesterov2006cubic} & Deterministic & $\mathcal{O}(\varepsilon^{-1.5})$ & - & - \\
        \cite{tripuraneni2018stochastic} & Stochastic & $\mathcal{O}(\varepsilon^{-1.5})$ & $\tilde{\mathcal{O}}(\varepsilon^{-3.5})$ & $\tilde{\mathcal{O}}(\varepsilon^{-2.5})$ \\
        \cite{zhou2020stochastic}$^\dagger$ & Stochastic & $\mathcal{O}(\varepsilon^{-1.5})$ & $\tilde{\mathcal{O}}(\varepsilon^{-3})$ & $\tilde{\mathcal{O}}(\varepsilon^{-2})$ \\
        \cite{chayti2025improving} & Stochastic & $\mathcal{O}(\varepsilon^{-3.5})$ & $\mathcal{O}(\varepsilon^{-3.5})$ & $\mathcal{O}(\varepsilon^{-3.5})$ \\
        \cite{yang2025faster}$^*$ & Stochastic & $\mathcal{O}(\varepsilon^{-2.5})$ & Accurate & $\mathcal{O}(\varepsilon^{-2.5})$ \\
        This work (Alg{.} 2 with $\beta = 0$)$^{**}$ & Stochastic & $\mathcal{O}(\varepsilon^{-1.5})$ & $\mathcal{O}(\varepsilon^{-3})$ & $\mathcal{O}(\varepsilon^{-2})$ \\
        \bottomrule
    \end{tabular}
    \vspace{0.1cm} 
    \caption{Summary of complexity results of cubic regularized Newton method for finding an $(\varepsilon, \sqrt{\varepsilon})$-second-order stationary point in the minimization setting. ``Iter.,'' ``Grad.,'' and ``Hess.'' stand for iteration complexity, (stochastic) gradient sample complexity, and (stochastic) Hessian sample complexity, respectively. 
    \endgraf 
    {\noindent\footnotesize
    $\dagger$ The theoretical analysis in \cite{zhou2020stochastic} contains a nontrivial unresolved technical gap.\\
    $*$ The approach in \cite{yang2025faster} requires near-accurate gradient information.\\
    $**$ By setting $\beta = 0$, Algorithm \ref{alg:stoch-cubic-minimax} naturally reduces to a variance-reduced stochastic CRN method for standard nonconvex minimization.
    \label{tab:minimization_complexity}
    }}
\end{table}

\vspace{0.2cm} 
\noindent\textbf{Other related works. }
In this paragraph, we briefly discuss related results on stochastic second-order methods for minimization problems,
as they are closely related to our technical byproduct. In this setting,
\cite{tripuraneni2018stochastic} used sample average approximations to estimate gradients
and Hessians within cubic regularized Newton methods, leading to
$\tilde{\mathcal{O}}(\varepsilon^{-3.5})$  gradient and 
$\tilde{\mathcal{O}}(\varepsilon^{-2.5})$ Hessian sample complexities to find an
$(\varepsilon,\sqrt{\varepsilon})$-second-order stationary point. Multiple works attempt to improve these complexities through momentum or variance reduction techniques \citep[etc.]{wang2019stochastic,zhou2020stochastic,zhang2022adaptive,chayti2025improving,yang2025faster}, among which \cite{zhou2020stochastic} achieves the state-of-the-art $\tilde{\mathcal{O}}(\varepsilon^{-3})$ gradient and
$\tilde{\mathcal{O}}(\varepsilon^{-2})$ Hessian sample complexities. They provided a nontrivial sharp observation on how the batch sizes in recursive
gradient and Hessian estimators should be adaptively coupled with the squared iterate displacement, which is beyond the standard variance reduction literature. Unfortunately, as discussed above, their analysis contains a critical issue, and has already influenced the theoretical rigor of a number of follow-ups. Our results close this gap by providing an alternative strategy of analysis.

\vspace{0.2cm}
\noindent\textbf{Organization. } 
The rest of this paper is organized as follows. Section \ref{sec:preliminaries} introduces the definitions and preliminary results. In Section \ref{sec:cubic-Newton}, we first present our ACQ majorization and introduce the resulting ACQ-regularized Newton method (ACQRN,  Algorithm \ref{alg:cubic-minimax}), we then establish its global and local convergence guarantees. Section \ref{sec:stochastic-cubic} introduces the stochastic variant of our method (S-ACQRN, Algorithm \ref{alg:stoch-cubic-minimax}) and establishes its convergence results. Section \ref{sec:numerical experiments} reports numerical experiments, demonstrating the empirical efficiency of the proposed approach. We conclude in the final section. \vspace{0.2cm}

\noindent\textbf{Notations. }  
For matrix $A$, we denote $\|A\|$ its operator (spectral) norm, and denote $\|A\|_F$ its Frobenius norm. For square matrix $A$, we denote $\lambda_{\min}(A)$ its minimal eigenvalue. $I_d$ denotes a $d\times d$ identity matrix. The transposed Jacobian of a mapping $F: \bb{R}^{d_1} \to \bb{R}^{d_2}$ is denoted as $\nabla F(v) \in \bb{R}^{d_1 \times d_2}$. Namely, let $F_i : \bb{R}^{d_1} \to \bb{R}$ be the $i$-th coordinate of $F$, then $\nabla F(v):= [\nabla F_1(v), ..., \nabla F_{d_2}(v)],$ where each $\nabla F_i(v)$ is a column vector in $\bb{R}^{d_1}$. For $w \in \mathbb{R}^{d_2}$, the second-order directional derivative $\nabla^2 F(v)[w]$ is defined by 
$\nabla^2 F(v)[w] := \sum_{i=1}^{d_2} w_i \nabla^2 F_i(v).$
For $f: \Rn \times \Rm \to \bb{R}$, we define its partial gradient, Hessian, and third (directional) derivative as $\nabla_x f(x,y) := \nabla g_{y}(x) \in \Rn$, $\nabla_{yx}^2 f(x,y) := \nabla g'_{x}(y) \in \bb{R}^{m \times n}$, and $\nabla_{xyx}^3 f(x,y)[w] := \sum_{i=1}^n w_i \nabla^2_{xy} g_{x,i}'(x,y) \in \bb{R}^{n \times m}$, respectively, with $g_{y}(x) = f(x, y)$, $g'_{x}(y) = \nabla_x f(x,y)$, and $g_{x,i}' = (\nabla_x f)_i$.
All the other partial gradients, Hessians, and third derivatives are defined in a similar manner. Finally, for simplicity, we denote $z:=(x,y)$ and $\zk := (\xk, \yk)$.

\section{Preliminaries} \label{sec:preliminaries}
In this section, we formally state our assumptions, introduce the equivalence between the first- and second-order optimality conditions for the minimax problem \eqref{Prob_Ori} and the regularized minimization problem \eqref{Prob_Min}, as well as a few other useful properties. Throughout this paper, we make the following assumptions on \eqref{Prob_Ori}, which are common assumptions for studying second-order methods in the minimax setting, see, e.g., \citep{luo2022finding,chen2023a}.

\begin{assumpt}
\label{Assumption_1}
The function $f$ is thrice continuously differentiable. For any fixed $x$, $f(x,\cdot)$ is $\mu$-strongly concave in $y$. For any fixed $y$, $f(\cdot,y)$ is possibly nonconvex in $x$. We require $f$ to have $L$-Lipschitz gradient and $\rho$-Lipschitz Hessian s.t. for  $\forall z,\xi \in \bb{R}^{n + m}$ we have
\[
    \|\df(z + \xi) - \df(z)\| \leq L\|\xi\| \qaq
    \|\nabla^2 f(z + \xi) - \nabla^2 f(z)\| \leq \rho \|\xi\|.
\]
By default, we assume the value function $\Phi(x)$ to be bounded below: $\bar \Phi := \inf_{x\in\Rn} \Phi(x) > -\infty.$
\end{assumpt}

For the value minimization reformulation \eqref{Prob_value_fcn}, let $\ys(x) := \mathop{\arg\max}_{y \in \Rm} f(x,y)$ denote the unique inner level maximizer. Then the value function $\Phi$ is twice continuously differentiable with $\nabla \Phi(x) = \nabla_x f(\zs)$ and
\begin{equation*}
\nabla^2\Phi(x) ={} \ddfxx(\zs) - \ddfxy(\zs)[\ddfyy(\zs)]^{-1}\ddfyx(\zs).
\end{equation*}
Under Assumption \ref{Assumption_1}, the regularized objective $\hb$ is twice continuously differentiable. For any $\beta>0$, we can compute the gradient and Hessian of $\hb$ by:
\[
\dhb(z)=\big[I_{n+m}+\beta \nabla^2 f(z)P\big]\df(z), \quad \text{where} \quad P:=\begin{pmatrix}0&0\\0&I_m\end{pmatrix},
\]
and $\nabla^2 \hb (z) = \Hb(z) + \Ht(z)$, where
\begin{equation}
\label{eq:Hess-components}
\Hb(z) := \nabla^2 f(z)+\beta \nabla^2 f(z)P\nabla^2 f(z) \qaq \Ht(z) := \beta\nabla^3 f(z)[P\df(z)].
\end{equation}

\noindent From these computations, it is clear that if one directly applies the cubic regularized Newton (CRN) method to $\hb$, it would require an additional Lipschitz assumption on the third-order tensor of derivatives of $f$ and expensive computation of these third-order derivatives. Fortunately, through our ACQ majorization scheme (in the next section), 
our ACQRN method only requires the first- and second-order information of $f$. The third-order derivatives in $\Ht$ are reserved only for theoretical analysis. Furthermore, we remark that the matrix $P$ is introduced primarily for notational convenience and practical computations should follow the formulas given in our previous work \citep{ma2026linesearch}.

Now, we present the definitions of first- and second-order minimax points of problem \eqref{Prob_Ori}, defined through those of the possibly nonconvex value function $\Phi$.

\begin{defin}
    \label{def:opt-cond-ori}
    We say that $(x^*, y^*)$ is a first-order minimax point of \eqref{Prob_Ori} if $\nabla \Phi(x^*) = 0$ and $y^* \in \ys(x^*)$, or equivalently, $\nabla f(x^*, y^*) = 0$. If additionally $\nabla^2\Phi(x^*) \succeq 0$, then we say that $(x^*, y^*)$ satisfies the second-order necessary condition (SONC) of \eqref{Prob_Ori}. If $\nabla^2\Phi(x^*) \succ 0$, then we say that $(x^*, y^*)$ satisfies the second-order sufficient condition (SOSC).
\end{defin} 


For the regularized minimization problem \eqref{Prob_Min}, the first- and second-order stationary points follow the conventions of nonconvex optimization for the joint minimization over $z = (x,y)$, which are stated below for presentation completeness.

\begin{defin}
    \label{def:opt-cond-min}
    We say that $z^* := (x^*, y^*)$ is a first-order stationary point of \eqref{Prob_Min} if $\nabla \hb(z^*) = 0$. If additionally $\nabla^2\hb(z^*) \succeq 0$, then we say that $z^*$ satisfies the SONC of \eqref{Prob_Min}. If $\nabla^2 \hb(z^*) \succ 0$,  then we say that $z^*$ satisfies the SOSC of \eqref{Prob_Min}.
\end{defin}

Finally, we introduce the notion of approximate second-order stationary points for  \eqref{Prob_Min}.

\begin{defin}
    \label{def:fosp-sosp-min}
    Let $\varepsilon, \delta \ge 0$. 
    We call $z^*$ an $(\varepsilon, \delta)$-SOSP (second-order stationary point) of \eqref{Prob_Min} if 
    $
     \|\dhb(z^*)\| \leq \varepsilon \text{ and } \nabla^2 \hb(z^*) \succeq -\delta I_{n+m}.
    $
\end{defin}

To establish iterate convergence and local superlinear convergence for our ACQRN method, we formally introduce the KL property \citep{lojasiewicz1965ensembles,kur98}, a mild assumption on the local geometry of the objective function. It holds for all (smooth) subanalytic and semialgebraic functions and a broad class of practical problems satisfy this property, see, e.g., \citep[Section 4]{AttBolRedSou10} and \citep[Section 5]{BolSabTeb14}.

\begin{defin}[KL property]
    \label{def:Loj-inequality}
    A differentiable function $F: \bb{R}^d \to \bb{R}$ is said to satisfy the KL property at a point $\bar x$ if there exists a neighborhood $U(\bar x)$ of $\bar x$ such that
    \[\hspace{-2ex}\|\nabla{F}(x)\|\geq C_F|F(x)-F(\bar x)|^\theta, \quad \forall~x \in U(\bar x),\]
    where $C_F = C_F(\bar x) >0$ and $\theta = \theta(\bar x) \in[\frac12,1)$ are constants depending on the point $\bar{x}$. In particular,  $\theta$ is called the 
    KL exponent of $F$ at $\bar x$.
\end{defin}

It is worth noting that the KL property holds for more general desingularization functions (see, e.g., \cite{AttBol09,AttBolSva13}) and Definition \ref{def:Loj-inequality} is sometimes called the {\L}ojasiewicz property in the literature. We will stick to Definition \ref{def:Loj-inequality} as it allows explicit asymptotic convergence rates. 

At the end of this section, we present a proposition characterizing the equivalence between \eqref{Prob_Ori} and \eqref{Prob_Min} in terms of lower boundedness, the first- and second-order stationary points, and the KL property, when $\beta > \muy^{-1}$. Since it follows directly from our previous paper \cite[Lemma 2.9, Theorem 3.3, 3.5, and 3.14]{ma2026linesearch}, we omit the proof of Proposition \ref{prop:equivalence}. 

\begin{prop}
    \label{prop:equivalence}
    Suppose Assumption \ref{Assumption_1} holds and $\beta > \muy^{-1}$. Then we have:  
    \begin{enumerate}[leftmargin=0.5cm, itemsep=1pt, topsep=0pt, parsep=0pt]
        \item \textbf{Lower bound equivalence:} the functions $h_\beta$ and $\Phi$ share the same minimal value. 
        \item \textbf{Optimality equivalence:} a point $(x^*, y^*)$ is a first-order minimax point of \eqref{Prob_Ori} iff. it is a first-order stationary point of \eqref{Prob_Min}; and the point $(x^*, y^*)$ satisfies the SONC (or SOSC) of \eqref{Prob_Ori} iff. it satisfies the SONC (or SOSC) of \eqref{Prob_Min}.
        \item \textbf{Local geometry equivalence:} The function $\Phi$ satisfies the KL property at $\bar x$ with exponent $\theta \in [\frac{1}{2},1)$ iff. $\hb$ satisfies the KL property at $(\bar x, \ys(\bar x))$ with the same exponent.
    \end{enumerate}
\end{prop}

\section{An adaptive cubic-quadratic regularized Newton method for \texorpdfstring{\eqref{Prob_Min}}{(RM)}}
\label{sec:cubic-Newton}
Given the preliminary discussion, in this section, we develop the deterministic single-loop framework for the minimax problem \eqref{Prob_Ori} based on the regularized reformulation (RM). We will start with the motivation and derivation of the adaptive cubic-quadratic (ACQ) majorization scheme, and then introduce the resulting adaptive cubic-quadratic regularized Newton (ACQRN) method and establish its global and local convergence results. 
\subsection{The ACQ majorization for regularized minimization}
\label{subsec:properties-of-h}
Before introducing the ACQ majorization, it is helpful to first investigate the cubic regularization scheme \citep{nesterov2006cubic} for $h_\beta$, which takes the form 
\[
\hat{h}_\beta(z;\xi):=h_\beta(z)+\nabla h_\beta(z)^\top \xi
+\frac{1}{2}\xi^\top \nabla^2 h_\beta(z)\xi
+\frac{M}{6}\|\xi\|^3.
\]
This is a standard second-order framework for nonconvex problems as it offers global convergence guarantees with favorable rates. Its line-search-free structure also makes it particularly suitable in stochastic settings. However, as shown in Section \ref{sec:preliminaries}, evaluating $\nabla^2 h_\beta(z)$ involves the computation of third-order derivatives, cf. \eqref{eq:Hess-components}. Moreover, the parameter $M$ should match the Lipschitz constant of $\nabla^2 h_\beta$, which, unfortunately, is impossible without additionally assuming the third- and zeroth-order Lipschitz continuity of $f$, while the latter conflicts with strong concavity in $y$. 

The above theoretical and computational weaknesses motivate us to derive a more flexible majorization scheme that is free from these bottlenecks. To begin with, we investigate the boundedness and Lipschitz continuity of the Hessian components $\Ht$ and $\Hb$, respectively. 

\begin{lem}
    \label{lem:surrogate-hessian-properties}
    Under Assumption \ref{Assumption_1} and $\beta > 0$, for any $z,\xi \in \bb{R}^{n+m}$, it holds that
    \[
    \|\Ht(z)\| \leq \beta\rho \|\dfy(z)\| \qaq \|\Hb(z + \xi) - \Hb(z)\| \leq (2\beta L + 1)\rho \|\xi\|.
    \]
\end{lem}
\begin{proof}
    
    Invoking \eqref{eq:Hess-components}, the triangular inequality, and the $\rho$-Lipschitz continuity of $\ddf$, yields
    \[
    \begin{aligned}
    &\|\Hb(z+\xi)-\Hb(z)\|\\
    \leq& \|\nabla^2 f(z+\xi)-\nabla^2 f(z)\|+\beta\|\nabla^2 f(z+\xi)P\nabla^2 f(z+\xi)-\nabla^2 f(z)P\nabla^2 f(z)\|\\
    \leq& \rho\|\xi\|
    +\beta\|[\nabla^2 f(z+\xi)-\nabla^2 f(z)]P\nabla^2 f(z+\xi)\|+\beta\|\nabla^2 f(z)P[\nabla^2 f(z+\xi)-\nabla^2 f(z)]\|\\
    \leq& \rho\|\xi\|
    +\beta\|\nabla^2 f(z+\xi)-\nabla^2 f(z)\| \cdot \|\nabla^2 f(z+\xi)\|+\beta\|\nabla^2 f(z)\| \cdot \|\nabla^2 f(z+\xi)-\nabla^2 f(z)\|\\
    \leq& (2\beta L+1)\rho\|\xi\|,
    \end{aligned}
    \]
    where the third inequality is due to $\|P\| \leq 1$ and the last line follows from the $L$-smoothness of $f$.
    Moreover, 
    using \eqref{eq:Hess-components} and the $\rho$-Lipschitz continuity of $\ddf$, we have
    \[
    \|\Ht(z)\|\leq \beta\rho\|P\df(z)\|=\beta\rho\|\dfy(z)\|.
    \]
    This completes the proof.
\end{proof}

With Lemma \ref{lem:surrogate-hessian-properties}, we are now ready to derive the ACQ majorization scheme. The core idea is to use the Lipschitz continuous component $\bar H$ of $\nabla^2 \hb$ as a surrogate Hessian, whereas 
an adaptive quadratic regularization term $\frac{\beta\rho}{2}\|\nabla_yf(z)\|\|\xi\|^2$ majorizes the non-Lipschitz component $\tilde H$ of the Hessian $\nabla^2 \hb$.
\begin{prop}
    \label{prop:descent-lemma}
    Under Assumption \ref{Assumption_1} and $\beta > 0$, for any $z, \xi \in \bb{R}^{n+m}$, it holds that
    \[
    \bigg|\hb(z + \xi) - \hb(z) - \dhb(z)^\top\xi - \frac{1}{2} \xi^\top \Hb(z) \xi \bigg| \leq \frac{(3\beta L+1)\rho}{6} \|\xi\|^3 + \frac{\beta \rho}{2} \|\dfy(z)\| \|\xi\|^2.
    \]
\end{prop}
\begin{proof}
    By the fundamental theorem of calculus applied to $t\mapsto h_\beta(z+t\xi)$, it holds that
    \begin{equation*}
    \begin{aligned}
    &\bigg|\hb(z + \xi) - \hb(z) - \dhb(z)^\top\xi - \frac{1}{2} \xi^\top \Hb(z) \xi \bigg|\\ 
    =&  \bigg|\int_0^1 \xi^\top [\dhb(z+t\xi) - \dhb(z)] dt - \frac{1}{2} \xi^\top \Hb(z)\xi \bigg|\\ 
    \overset{(a)}{=}& \bigg|\int_0^1 \bigg\{\int_0^1 \xi^\top [\tilde H(z+st\xi) + \Hb(z + st\xi) ] \xi ds \bigg\} t dt - \frac{1}{2} \xi^\top \Hb(z)\xi \bigg|\\
    =& \bigg|\int_0^1 \bigg\{\int_0^1 \xi^\top [\tilde H(z+st\xi) + \Hb(z + st\xi) -\Hb(z) ] \xi ds \bigg\} t dt \bigg|\\
    \overset{(b)}{\leq}& \|\xi\|^2 \int_0^1 \bigg\{\int_0^1 \|\tilde H(z+st\xi)\| + \|\Hb(z + st\xi) -\Hb(z) \| ds \bigg\} t dt\\
    \overset{(c)}{\leq}&
    \|\xi\|^2 \int_0^1 \bigg\{\int_0^1 \beta\rho\|\dfy(z+st\xi)\| + (2\beta L + 1)\rho \|\xi\| st ds \bigg\} t dt
    \\
    \overset{(d)}{\leq}&
    \|\xi\|^2 \int_0^1 \bigg\{\int_0^1 \beta\rho[\|\dfy(z)\| + L\|\xi\|st] + (2\beta L + 1)\rho \|\xi\| st ds \bigg\} t dt
    \\
    =&
    \frac{\beta\rho}{2} \|\xi\|^2 \|\dfy(z)\| + \frac{(3\beta L+1)\rho}{6} \|\xi\|^3,
\end{aligned}
\end{equation*}
where line (a) follows, again,  from the fundamental theorem of calculus applied to
$s\mapsto \nabla h_\beta(z+st\xi)$ and the fact that
$\nabla^2 h_\beta=\tilde H+\bar H$, (b) is due to the triangle inequality, (c) is due to Lemma \ref{lem:surrogate-hessian-properties}, and (d) is due to the $L$-smoothness of $f$.
\end{proof}

At the end of this subsection, we present a lemma on the first-order approximation of $\dhb$ using the surrogate Hessian $\Hb$.

\begin{lem}
    \label{lem:first-order-approx-of-grad}
    Under Assumption \ref{Assumption_1} and $\beta > 0$, for any $z, \xi \in \bb{R}^{n+m}$, it holds that
    \[
    \|\nabla \hb(z + \xi) - \nabla \hb(z) - \Hb(z) \xi\| \leq \beta\rho \|\dfy(z)\|\|\xi\| + \frac{3\beta L+1}{2} \rho \|\xi\|^2.
    \]
\end{lem}
\begin{proof}
Applying the fundamental theorem of calculus to $t\mapsto \nabla h_\beta(z+t\xi)$ yields 
\begin{equation*}
    \begin{aligned}
    \|\nabla \hb(z + \xi) - \nabla &\hb(z) - \Hb(z) \xi\|
    ={} \bigg\|\int_0^1 [\tilde H(z+t\xi) + \Hb(z+t\xi) - \Hb(z)] \xi dt\bigg\|\\
    \leq{}&
    \int_0^1 \beta\rho\|\dfy(z+t\xi)\|\|\xi\| + (2\beta L + 1)\rho \|\xi\|^2t dt
    \\
    \leq{}&
    \int_0^1 \beta\rho[\|\dfy(z)\|\|\xi\| + Lt\|\xi\|^2]+ (2\beta L + 1)\rho \|\xi\|^2t dt
    \\
    ={}&
    \beta\rho \|\dfy(z)\|\|\xi\| + \frac{3\beta L+1}{2} \rho \|\xi\|^2,
    \end{aligned}
\end{equation*}    
where the first inequality follows from the triangle inequality and Lemma \ref{lem:surrogate-hessian-properties}, and the second inequality is due to the $L$-smoothness of $f$.
\end{proof}

\subsection{A single-loop ACQRN method}
\label{subsec:deter-CRN-method}
In this subsection, we formally introduce our ACQRN method in Algorithm \ref{alg:cubic-minimax} for the minimax problem \eqref{Prob_Ori}. In each iteration $k$, the update takes the form $\zkp = \zk + \xik$, where $\xik$ is a solution of the subproblem \eqref{eq:cubic-subproblem} based on the ACQ majorization. Notably, since $\Hb$ only involves the Hessian of $f$, our cubic regularized subproblem avoids the non-Lipschitz third-order term $\tilde H$ in the full Hessian of $\hb$. 
Hence, it is computationally comparable to the cubic regularized subproblem in the standard minimization setting. As for the parameter choices, by Proposition \ref{prop:equivalence}, we require $\beta > \muy^{-1}$ to guarantee the equivalence between the minimax problem \eqref{Prob_Ori} and the regularized minimization \eqref{Prob_Min}. Furthermore, by the ACQM scheme (Proposition \ref{prop:descent-lemma}), when the regularization parameters $\alo, \alt > 0$ are sufficiently large, the function value $\hb(\zk)$ has guaranteed descent in every iteration. We will justify the specific choices of $\alo, \alt$ later. 

\begin{algorithm}[ht]
\caption{Adaptive cubic-quadratic regularized Newton method for minimax problems}
\label{alg:cubic-minimax}
\begin{algorithmic}[1]
\State{\textbf{Initialization:} Initialize $z_0$ and parameters $\beta > \muy^{-1}$, $\alo = 2\beta\rho$ and $\alt =2(3\beta L+1)\rho$.}
\For{$k = 0, 1, 2, \ldots$}
\State{Solve the ACQ majorized subproblem
\begin{equation}
    \label{eq:cubic-subproblem}
    \xik \in \argmin_{\xi \in \bb{R}^{n+m}} \dhb(z_k)^\top\xi + \frac{1}{2} \xi^\top \Big(\Hb(z_k) + \alo \|\dfy(z_k)\| I_{n+m}\Big) \xi + \frac{\alt}{6} \|\xi\|^3.
\end{equation}}
\vspace{-3mm}
\State{$\zkp \gets \zk + \xik$} 
\EndFor
\end{algorithmic}
\end{algorithm} 

In the rest of this section, we present some preliminary analysis of  Algorithm \ref{alg:cubic-minimax}. Instead of directly analyzing the subproblem \eqref{eq:cubic-subproblem}, we study the following general subproblem:
\begin{equation}
    \label{eq:cubic-subproblem-general}
    \begin{aligned}
        \xi^+ = \argmin_{\xi \in \bb{R}^{n + m}} g^\top \xi + \frac{1}{2}\xi^\top \Big(H + \alo \|g_y\|I_{n+m}\Big)\xi + \frac{\alt}{6}\|\xi\|^3 \qaq
        z^+ = z + \xi^+,
    \end{aligned}
\end{equation}
which allows $g, g_y$, and $H$ to deviate from the problem data $\dhb(z)$, $\dfy(z)$, and $\Hb(z)$, respectively. This prepares us for the analysis under the stochastic setting in the subsequent section. We begin with the (global) optimality conditions for the ACQ regularized subproblem \eqref{eq:cubic-subproblem-general}, which follows the results of \citep{nesterov2006cubic,cartis2011adaptivea}. 

\begin{lem}[{\citealp{nesterov2006cubic}}]
    \label{lem:opt-cond-cubic-subp}
    Suppose Assumption \ref{Assumption_1} holds,  $\beta > 0$, and $H$ is symmetric, then $\xi^+$ is a solution of the cubic regularized subproblem \eqref{eq:cubic-subproblem-general} if and only if
    \[\begin{aligned}
    g + \Big\{H + \Big[\alo \|g_y\| + \frac{\alt}{2} \|\xi^+\|\Big] I_{n+m}\Big\}\xi^+ = 0
    \qaq
    H + \Big[\alo \|g_y\| + \frac{\alt}{2} \|\xi^+\| \Big] I_{n+m} \succeq 0.
    \end{aligned}\]
\end{lem}

Motivated by the ACQ scheme (Proposition \ref{prop:descent-lemma}) and the optimality conditions of the subproblem \eqref{eq:cubic-subproblem-general}, we have the following descent property in function value. It is worth noting that the error terms $\|\dhb(z) - g\|$, $\|\bar H(z) - H\|$, and $\|g_y - \dfy(z)\|$ on the RHS vanish in the deterministic setting, when the gradient and Hessian information of $f$ are accurate. Furthermore, Lemma \ref{lem:func-value-descent} implies that as long as $\alo > \beta\rho$ and $\alt > 4(3\beta L+1)\rho/3$, Algorithm \ref{alg:cubic-minimax} has guaranteed descent in each iteration. 

\begin{lem}
    \label{lem:func-value-descent}
    Suppose Assumption \ref{Assumption_1} holds,  $\beta > 0$, and $H$ is symmetric. Let $\xi^+$ and $z^+$ be as given in \eqref{eq:cubic-subproblem-general}. Then, setting $c:=(3\beta L+1)\rho/6$, it holds that
    \begin{equation*}
        \begin{aligned}
        \hb(z^+) - \hb(z)
        &\leq 
        -\Big(\frac{\alo}{2} - \frac{\beta\rho}{2}\Big) \|\dfy(z)\| \|\xi^+\|^2 - \Big(\frac{\alt}{4} - 2c \Big) \|\xi^+\|^3 + \\
        &\hspace{-1.5cm}  
        \frac{2}{3\sqrt{c}}\|\dhb(z) - g\|^{\frac32}
        + \frac{1}{6c^2} \|\bar H(z) - H\|^3 + \frac{\alo^3}{6c^2} \|\dfy(z) - g_y\|^3.
        \end{aligned}
    \end{equation*}
\end{lem}
\begin{proof}
    By Proposition \ref{prop:descent-lemma} and the equality in Lemma \ref{lem:opt-cond-cubic-subp}, it holds that
    \begin{equation*}
    \begin{aligned}
    \hb(z^+& )  - \hb(z) 
    \leq{} \dhb(z)^\top \xi^+ + \frac{1}{2} (\xi^+)^\top \Hb(z) \xi^+ + \frac{\beta\rho}{2}\|\dfy(z)\| \|\xi^+\|^2 + c \|\xi^+\|^3 \\
     & - (\xi^+)^\top \Big(g + H \xi^+ + \alo \|g_y\| \xi^+ + \frac{\alt}{2} \|\xi^+\|\xi^+\Big)\\
    ={}& [\dhb(z)-g]^\top \xi^+ + \frac{1}{2} (\xi^+)^\top [\Hb(z)-H] \xi^+ + \frac{\beta\rho}{2}\|\dfy(z)\|   \|\xi^+\|^2 -\\
    & \frac{1}{2} (\xi^+)^\top \Big[H + \Big(\alo \|g_y\| + \frac{\alt}{2} \|\xi^+\|\Big)I_{n+m}\Big]\xi^+ - \frac{\alo}{2} \|g_y\| \|\xi^+\|^2 -   \Big(\frac{\alt}{4} - c\Big)\|\xi^+\|^3\\
    \leq{}& \|\dhb(z)-g\|   \|\xi^+\| + \frac{1}{2} \|\Hb(z)-H\|   \|\xi^+\|^2 + \frac{\alo}{2} \|\dfy(z) - g_y\|  \|\xi^+\|^2 -\\
    & \Big(\frac{\alo}{2} - \frac{\beta \rho}{2}\Big) \|\dfy(z)\| \|\xi^+\|^2 -  \Big(\frac{\alt}{4} - c\Big)\|\xi^+\|^3,
    \end{aligned}        
    \end{equation*}
    where the last line follows from the Cauchy-Schwarz inequality, the reverse triangle inequality and the inequality in Lemma \ref{lem:opt-cond-cubic-subp}. To complete the proof, we utilize the following variant of Young's inequality \[
    ab \leq {p^{-1}(a\theta)^p} + {q^{-1}(b/\theta)^q}, \quad \forall a,b,p,q,\theta > 0 \quad \text{satisfying} \quad p^{-1} + q^{-1} = 1,
    \]
    and get 
    \begin{align*}
        \|\dhb(z)-g\|   \|\xi^+\| &\leq \frac{c}{3} \|\xi^+\|^3 + \frac{2}{3\sqrt{c}} \|\dhb(z)-g\|^{\frac32};\\
        \|\Hb(z)-H\|   \|\xi^+\|^2 &\leq \frac{2c}{3} \|\xi^+\|^3 + \frac{1}{3c^2} \|\Hb(z) - H\|^3;\\
        \|\dfy(z) - g_y\|  \|\xi^+\|^2 &\leq \frac{2c}{3\alo} \|\xi^+\|^3 + \frac{\alo^2}{3c^2} \|\dfy(z) - g_y\|^3.
    \end{align*}
    Combining the previous estimates yields the desired result.
\end{proof}

Next, we establish an upper bound of the gradient size $\|\dhb(z^+)\|$ and  a lower bound of the minimal eigenvalue of the Hessian $\nabla^2 \hb(z^+)$. Due to our adaptive majorization scheme, the bounds additionally involve the partial gradient norm $\|\dfy(z)\|$.

\begin{lem}
    \label{lem:upper-bound-of-h-gradient}
    Under the same assumption and parameter settings of Lemma \ref{lem:func-value-descent}, we have 
    \begin{align*}
        \|\dhb(z^+)\| 
        &\leq  (\beta\rho + \alo)\|\dfy(z)\|   \|\xi^+\| + (4c+\alt/2) \|\xi^+\|^2 + \\
        &\hspace{1cm} \|\dhb(z) - g\| + (2c)^{-1} \|\bar H(z) - H\|^2 + \alo^2 (2c)^{-1} \|\dfy(z) - g_y\|^2.
    \end{align*}
\end{lem}
\begin{proof}
    By the triangle inequality, it holds that
    \[\begin{aligned}
    \|\dhb(z^+)\|
    \leq& \|\dhb(z^+) - \dhb(z) - \Hb(z)\xi^+\| + \|\dhb(z)-g\|+\\
    &\hspace{1cm}\|[\Hb(z)-H]\xi^+\|+\|g+H\xi^+\|\\
    \leq& \beta\rho\|\dfy(z)\|\|\xi^+\| + 3c\|\xi^+\|^2 +  \|\dhb(z)-g\|+\|\Hb(z)-H\| \|\xi^+\|+\\
    &\hspace{1cm} \alo \|g_y\| \|\xi^+\| + \alt\|\xi^+\|^2/2\\
    \leq& (\beta\rho + \alo)\|\dfy(z)\|   \|\xi^+\| + (3c + \alt/2) \|\xi^+\|^2 +  \|\dhb(z) - g\| + \\
    &\hspace{1cm} \|\Hb(z) - H\|   \|\xi^+\| + \alo \|\dfy(z)-g_y\|   \|\xi^+\|
    \end{aligned}\]
    where the second inequality is due to Lemma \ref{lem:first-order-approx-of-grad} and the equality in Lemma \ref{lem:opt-cond-cubic-subp} and the last inequality utilizes the triangle inequality.
    Then, invoking Young's inequality gives 
    \begin{align*}
        2\|\Hb(z) - H\|   \|\xi^+\| &\leq c^{-1} \|\Hb(z) - H\|^2 + c\|\xi^+\|^2;\\
        2\alo\|\dfy(z)-g_y\|   \|\xi^+\| &\leq \alo^2c^{-1} \|\dfy(z)-g_y\|^2 + c \|\xi^+\|^2.
    \end{align*}
    Combining the previous estimates yields the desired result.
\end{proof}

\begin{lem}
    \label{lem:lower-bound-surrogate-hessian}
    Under the same assumption and parameter settings of Lemma \ref{lem:func-value-descent}, we have
    \[
    \begin{aligned}
        \lambda_{\min}(\nabla^2 \hb (z^+)) &\geq -[(3\beta L + 1)\rho + \alt/2]\|\xi^+\| - (\beta\rho+\alo) \|\dfy(z)\| - \\
        &\hspace{1.5cm}
        \alo \|\dfy(z) - g_y\| - \|\Hb(z) - H\|.
    \end{aligned}
    \]
\end{lem}
\begin{proof}
By the inequality in Lemma \ref{lem:opt-cond-cubic-subp} and the triangle inequality, it holds that
\[
\lambda_{\min}(H) \geq - \alo \|g_y\| - \frac{\alt}{2} \|\xi^+\| \geq  - \alo \|g_y - \dfy(z)\| - \alo \|\dfy(z)\| - \frac{\alt}{2} \|\xi^+\|.
\]
Furthermore, by Lemma \ref{lem:surrogate-hessian-properties}, we have
\[
\begin{aligned}
    \lambda_{\min}(\bar H(z^+) - \bar H(z) + \tilde H(z^+)) &\geq -(2\beta L + 1)\rho \|\xi^+\| - \beta\rho\|\dfy(z^+)\|\\
    &\geq -(3\beta L + 1)\rho \|\xi^+\| - \beta\rho\|\dfy(z)\|,
\end{aligned}
\]
where the second inequality is due to triangle inequality and $L$-smoothness of $f$.
Since \[\nabla^2 \hb(z^+)
    = H + \bar H(z^+) - \bar H(z) + \tilde H(z^+) + \Hb(z) - H,\]
combining the previous estimates completes the proof.
\end{proof}

In the last technical lemma of this section, we establish the relationship between $\|\dfy(z)\|$ and $\|\xi^+\|$, when $\beta > \muy^{-1}$.

\begin{lem}
    \label{lem:relationship-dfy-xik}
    Under the same assumption and parameter settings of Lemma \ref{lem:func-value-descent}, and if in addition, we require $\beta>\mu^{-1}$, then it holds that  
    \[\begin{aligned}
    \|\dfy(z)\| &\leq L\|\xi^+\| + (\beta \muy -1)^{-1} \Big[(\beta\rho + \alo)\|\dfy(z)\|   \|\xi^+\| + (4c+\alt/2) \|\xi^+\|^2 + \\
    &\hspace{0.5cm} \|\dhb(z) - g\| +  (2c)^{-1} \|\bar H(z) - H\|^2 + \alo^2(2c)^{-1} \|\dfy(z) - g_y\|^2\Big].
    \end{aligned}\]
\end{lem}
\begin{proof}
    By the triangle inequality and $L$-smoothness of $f$, we have
    \[\begin{aligned}
    (\beta \muy -1) \|\dfy(z)\| 
    \leq& (\beta \muy -1) \|\dfy(z^+)\| + (\beta \muy -1) L\|\xi^+\| \\
    \leq& \|\dhb(z^+)\| + (\beta \muy -1)L\|\xi^+\|\\
    \leq& (\beta\rho + \alo)\|\dfy(z)\|   \|\xi^+\| + (4c+\alt/2) \|\xi^+\|^2 +  \|\dhb(z) - g\| + \\
    &\hspace{0.5cm} (2c)^{-1} \|\bar H(z) - H\|^2 + \alo^2(2c)^{-1} \|\dfy(z) - g_y\|^2 + L(\beta \muy -1)\|\xi^+\|,
    \end{aligned}\]
    where the second inequality follows from \cite[Lemma 2.13]{ma2026linesearch} and the last inequality is due to Lemma \ref{lem:upper-bound-of-h-gradient}.
    Dividing both sides by $\beta \muy -1 > 0$ completes the proof.
\end{proof}

\subsection{Global Convergence}
\label{subsec:global-conv}
Building on the previous preparation, we establish the global convergence of Algorithm \ref{alg:cubic-minimax}, and the corresponding sublinear rates.

\begin{theo}
    \label{thm:global-conv-cubic}
    Suppose Assumption \ref{Assumption_1} holds. Let $c := (3\beta L+1)\rho/6$ and $\Delta_0 := \hb(z_0) - \bar \Phi$. Then, the iterates $\{z_k\}_k$ generated by Algorithm \ref{alg:cubic-minimax}
    satisfy
    \[
    \hb(\zk) \to h^*,\quad \dhb(\zk) \to 0, \qaq \liminf_{k\to\infty}\lambda_{\min}(\nabla^2 \hb(\zk)) \geq 0,
    \]
    where $h^* \geq \bar \Phi$.
    Furthermore, for any $T \geq 1$, it holds that
    \[
    \min_{k = 1,\ldots, T} \max\bigg\{ \sqrt{\frac{\|\dhb(\zk)\|}{c}}, -\frac{\lambda_{\min}(\nabla^2 \hb(\zk))}{c} \bigg\} = O\bigg(\bigg[\frac{\Delta_0}{c T}\bigg]^{\frac13}\bigg).
    \]
\end{theo}

\begin{proof}
    Since we work with exact gradients $\dhb(\zk)$, $\dfy(\zk)$ and Hessian component $\Hb(\zk)$ in Algorithm \ref{alg:cubic-minimax}, the error terms in Lemma \ref{lem:func-value-descent} -- \ref{lem:relationship-dfy-xik} vanish. Then, replacing the regularization parameters $\alo =2\beta\rho$ and $\alt =12c$ with their values in Lemma \ref{lem:func-value-descent} -- \ref{lem:relationship-dfy-xik} yields
    \begin{equation}
    \label{eq:global-conv-inter0}
    \begin{aligned}
        c \|\xik\|^3 + \frac{\beta \rho}{2} \|\dfy(\zk)\|\|\xik\|^2 &\leq \hb(\zk) - \hb(\zkp), \\
         \|\dhb(\zkp)\| &\leq 10c \|\xik\|^2 + 3\beta \rho\|\dfy(\zk)\|\|\xik\|,\\
         -\lambda_{\min}(\nabla^2 \hb(\zkp)) &\leq  12c \|\xik\| + 3\beta \rho \|\dfy(\zk)\|,\\
         \|\dfy(z_k)\| &\leq L\|\xik\| + (\beta \muy-1)^{-1} \big[10c \|\xik\|^2 + 3\beta\rho\|\dfy(\zk)\|\|\xik\|\big],\\
         \|\dfy(z_k)\|\|\xik\| &\leq L\|\xik\|^2 + (\beta \muy-1)^{-1} \big[10c \|\xik\|^3 + 3\beta\rho\|\dfy(\zk)\|\|\xik\|^2\big],
    \end{aligned}
    \end{equation}
    where we multiply both sides of the fourth inequality by $\|\xik\|$ to get the last one. Now, since $\beta>\muy^{-1}$, Proposition \ref{prop:equivalence} implies that $\hb(\zk) \geq \bar \Phi$ for all $k$. By the first inequality in \eqref{eq:global-conv-inter0}, $\{\hb(\zk)\}_k$ is nonincreasing and thus there exists $h^* \geq \bar \Phi$ such that $\hb(\zk)\to h^*$. Moreover, telescoping this inequality over $k = 0, 1,\ldots, T-1$ yields
    \begin{equation}
    \label{eq:global-conv-inter1}
    c {\sum}_{k=0}^{T-1} \|\xik\|^3 + \frac{\beta \rho}{2} {\sum}_{k=0}^{T-1} \|\dfy(\zk)\|\|\xik\|^2 \leq \hb(z_0) - \bar \Phi =: \Delta_0.
    \end{equation}
    Since the RHS does not depend on $T$, taking $T \to \infty$ implies $\xik \to 0$ and $\|\dfy(\zk)\|\|\xik\|^2 \to 0$. 
    Combining these two convergence results with \eqref{eq:global-conv-inter0} (from bottom to top) yields 
    $\dfy(\zk) \to 0$, $\dhb(\zk) \to 0$, and $\liminf_{k\to\infty} \lambda_{\min} (\nabla^2 \hb(z_k)) \geq 0$.
    
    To establish the finite-step convergence rate result, observe that \eqref{eq:global-conv-inter1} implies
    \[
    \min_{k = 0,1,\ldots,T-1} \Big\{c \|\xik\|^3 + \frac{\beta \rho}{2} \|\dfy(\zk)\|\|\xik\|^2\Big\} \leq \frac{\Delta_0}{T}.
    \]
    Let the minimum on the LHS be attained at iteration $k^* \leq T - 1$. Then we obtain that 
    $\|\xi_{k^*}\| \leq ({\Delta_0}/{cT})^{1/3}$ and $(\beta \rho/2) \|\dfy(z_{k^*})\|\|\xi_{k^*}\|^2 \leq {\Delta_0}/{T}$. Then, repeatedly plugging them into \eqref{eq:global-conv-inter0} from bottom to top proves the results.  In detail, we first plug these bounds into the last inequality of \eqref{eq:global-conv-inter0} to yield 
    $$\|\dfy(z_{k^*})\|\|\xi_{k^*}\|  = O\bigg(L \bigg[\frac{\Delta_0}{cT}\bigg]^{\frac23}+ \frac{1}{\beta\muy-1}\cdot\frac{\Delta_0}T\bigg).$$
    Next, substituting this inequality into the second last inequality of \eqref{eq:global-conv-inter0} bounds $\|\dfy(z_{k^*})\|$ by
    $$\|\dfy(z_{k^*})\| = O\bigg(L\bigg[\frac{\Delta_0}{cT}\bigg]^{\frac13} + \frac{c}{\beta\muy-1}\cdot\bigg[\frac{\Delta_0}{cT}\bigg]^{\frac23} + \frac{\beta\rho}{(\beta\muy-1)^2}\cdot\frac{\Delta_0}{T}\bigg).$$
    Finally, substituting the above upper bounds on $\|\xi_{k^*}\|$, $\|\dfy(z_{k^*})\|\|\xi_{k^*}\|^2$, and $\|\dfy(z_{k^*})\|$ to the second and third inequalities of \eqref{eq:global-conv-inter0} provides the final results:
    \[
    \begin{aligned} 
         \|\dhb(z_{k^*+1})\| &= O\bigg(c \bigg[\frac{\Delta_0}{cT}\bigg]^{\frac23} + \frac{\beta\rho}{\beta\muy-1}\frac{\Delta_0}{T}\bigg)=O\bigg(c \bigg[\frac{\Delta_0}{cT}\bigg]^{\frac23}\bigg),\\
         -\lambda_{\min}(\nabla^2 \hb(z_{k^*+1})) &\leq  O\bigg(c \bigg[\frac{\Delta_0}{cT}\bigg]^{\frac13} + \frac{c\beta\rho}{\beta\muy-1}\bigg[\frac{\Delta_0}{cT}\bigg]^{\frac23} + \frac{(\beta\rho)^2}{(\beta\muy-1)^2}\frac{\Delta_0}{T}\bigg)=O\bigg(c \bigg[\frac{\Delta_0}{cT}\bigg]^{\frac13}\bigg),
    \end{aligned}
    \]
    where we repeatedly utilize the fact that $\beta L \rho/2 \leq c$. This completes the proof.
\end{proof}

By Proposition \ref{prop:equivalence} and Theorem \ref{thm:global-conv-cubic}, every accumulation point of Algorithm \ref{alg:cubic-minimax} satisfies the SONC of the minimax problem \eqref{Prob_Ori}. Furthermore, denoting $\kappa = L/\mu$ as the condition number of the inner maximization problem, this global convergence theorem immediately yields the following iteration complexity result.

\begin{coro}
    \label{coro:complexity-deter}
    Suppose Assumption \ref{Assumption_1} holds and let $\Delta_0 = \hb(z_0) - \bar \Phi$. Then, Algorithm \ref{alg:cubic-minimax} with $\beta=2\muy^{-1}$
    finds an $(\varepsilon, \sqrt{\kappa\rho\varepsilon})$-SOSP of \eqref{Prob_Min} with at most
    $O(\sqrt{\kappa \rho}\Delta_0\varepsilon^{-3/2})$
    iterations.
\end{coro}

A remark about Theorem \ref{thm:global-conv-cubic} and Corollary \ref{coro:complexity-deter} is that, aside from the dependency on $\varepsilon$, our iteration complexity is not directly comparable with that in \citep{luo2022finding} (or \citep{chen2026homogeneous}), which utilizes a double-loop CRN (or homogeneous descent) method to solve \eqref{Prob_value_fcn}, which is a different minimization reformulation of the minimax problem \eqref{Prob_Ori}. Translating iteration complexity results between these two reformulations maintains the dependency on $\varepsilon$, but  typically incurs a loss regarding the dependency on the problem constants $\rho$ and $\kappa$. Therefore, it will be unfair to claim the $O(\sqrt{\kappa})$ dependence in our result to be superior to the $O(\kappa^{3/2})$ dependence in double-loop methods.

Nevertheless, for practical problems, when the iterates approach a nondegenerate local solution, a single-loop second-order method often better exploits the local geometry of the problem. Especially when entering the local superlinear convergence region, its local behavior is essentially independent of the problem conditioning. In contrast, double-loop methods remain affected by  the problem's conditioning due to the inner-loop bottleneck,  even when the local geometry is already favorable to second-order methods. This motivates our local convergence analysis in the next subsection. In particular, we also empirically evaluate this phenomenon in Section \ref{sec:numerical experiments}, where the curve of running time versus condition number for our method grows much slower than for double-loop methods. 

\subsection{Fast Local Convergence under Regularity Conditions}
\label{subsec:iterate-conv-KL}
In this subsection, we establish iterate convergence (full sequence convergence) and fast local convergence rates of the ACQRN method under favorable local geometry, including the KL property, and the standard nondegeneracy assumption that the sequence enters a proper neighborhood of local solution with positive definite Hessian.     

\begin{assumpt}
    \label{ass:value-function-loj}
    The function $h_\beta$ with $\beta>\mu^{-1}$ satisfies the KL property at each point in  $\crit(\hb) := \{(x,y) \in \Rn\times\Rm: \dhb(x,y) = 0, \nabla^2 \hb(x,y) \succeq 0\}$  with exponent in $[1/2, 1)$.
\end{assumpt}

By Proposition \ref{prop:equivalence}, when $\beta > \muy^{-1}$, Assumption \ref{ass:value-function-loj} is equivalent to assuming the value function $\Phi$ to satisfy the KL property on $\crit(\Phi):=\{x\in\Rn:  \nabla\Phi(x) = 0, \nabla^2 \Phi(x) \succeq 0\}$, which is a standard assumption in value minimization based literature. This allows us to establish the iterate convergence and fast local convergence.

\begin{theo}
    \label{thm:local-rates-KL}
    Suppose Assumptions \ref{Assumption_1} and \ref{ass:value-function-loj} hold. Moreover, suppose the iterates $\{\zk\}_k$ generated by Algorithm \ref{alg:cubic-minimax} with $\beta > \muy^{-1}$, $\alo = 2\beta\rho$, and $\alt = 2(3\beta L+1)\rho$ have at least one accumulation point $z^* = (x^*, y^*)$. Let $\theta$ be the KL exponent of $\hb$ at $z^*$. The following statements are true. 
    \begin{enumerate}[leftmargin=0.5cm, itemsep=1pt, topsep=0pt, parsep=0pt]
        \item If $z^*$ satisfies the SOSC of \eqref{Prob_Ori}, 
        $\|\zk - z^*\|$ converges (locally) Q-quadratically to 0.
        \item If $\theta \in [1/2, 2/3)$, 
        $\|\zk - z^*\|$ converges (locally) Q-superlinearly at a rate of ${2}/{(3\theta)}$ to 0.
        \item If $\theta = 2/3$, 
        $\|\zk - z^*\|$ converges (locally) Q-linearly to 0.
        \item If $\theta \in (2/3,1)$, then
        $
        \|\zk - z^*\| = O(k^{-2(1-\theta)/(3\theta - 2)}).
        $
    \end{enumerate}
\end{theo}

\begin{proof}
From the fourth inequality in \eqref{eq:global-conv-inter0}, 
we have
\[
\|\dfy(\zk)\| \leq L\|\xik\| + (\beta\muy-1)^{-1}\big[10c\|\xik\|^2 + 3\beta\rho\|\dfy(\zk)\|\|\xik\|\big].
\]
Since $\|\xik\| \to 0$, for all sufficiently large $k$ the term $3\beta\rho(\beta\muy-1)^{-1}\|\xik\| \leq 1/2$, and rearranging gives $\|\dfy(\zk)\| \leq 2L\|\xik\| + O(\|\xik\|^2) = O(\|\xik\|)$.

As a result, for all sufficiently large $k$, by \eqref{eq:global-conv-inter0}, there exist $b_1, b_2 > 0$ such that
\begin{equation}
    \label{eq:local-conv-inter1}
    b_1\|\xik\|^3 \leq \hb(\zk) - \hb(\zkp) \qaq
    \|\dhb(\zkp)\| \leq b_2 \|\xik\|^2,
\end{equation}
which, together with continuity of $\hb$, verifies \citep[Conditions (H1'), (H2), and (H3)]{li2024generalized}. 
For the first bullet point, since $z^*$ satisfies the SOSC of \eqref{Prob_Ori}, by Proposition \ref{prop:equivalence}, it also satisfies the SOSC of \eqref{Prob_Min}. Then, the local Q-quadratic convergence of $z_k$ to $z^*$ follows from \citep[Theorem 4.4 and Corollary 6.5]{li2024generalized}. Furthermore, under the KL property, \citep[Theorem 5.2 and Remark 5.3]{li2024generalized} imply that $\zk \to z^*$ with the corresponding convergence rates quoted in this theorem. 
\end{proof}



\section{A stochastic ACQRN method for minimax problems}
\label{sec:stochastic-cubic}
In this section, we investigate the following stochastic setting of the minimax problem
\begin{equation}
    \tag{MM-Sto}
    \label{prob_stoch}
    \min_{x\in \Rn}\max_{y\in \Rm} f(x, y) = \Exp_{\zeta\sim Z}[F(x, y; \zeta)],
\end{equation}
and present a single-loop stochastic ACQRN (S-ACQRN) method based on the regularized minimization reformulation \eqref{Prob_Min}.
In addition to Assumption \ref{Assumption_1}, we require the following assumption on $F$. It is worth noting that our Assumption \ref{ass:stoch} is significantly weaker than those in \citep{chen2023a}, which require the boundedness and smoothness of $\dF$ and $\ddF$ to hold almost surely. In particular, we would like to emphasize that the global boundedness of $\dF$ is not compatible with the strong concavity assumption in terms of the variable $y$. In this section, we follow our convention and denote $z := (x,y)$.
\begin{assumpt}
    \label{ass:stoch} 
    For all $z \in \bb{R}^{n+m}$, when $\zeta\sim Z$, 
	the stochastic gradient $\dF(z;\zeta)$ and Hessian $\ddF(z;\zeta)$ are unbiased estimators of $\df(z)$ and $\ddf(z)$, respectively. That is,
    \[\Exp\big[\dF(z;\zeta)\big] = \df(z) \qaq \Exp\big[\ddF(z;\zeta)\big] = \ddf(z).\]
    Furthermore, for any $z, z_1, z_2 \in \bb{R}^{n+m}$, it holds that
    \[\begin{aligned}
    \Exp\big[\|\dF(z;\zeta)-\df(z)\|^3\big] &\leq \sigo^3, &\quad
    \Exp\big[\|\dF(z_1; \zeta) - \dF(z_2; \zeta)\|^3\big] &\leq L^3 \|z_1 - z_2\|^3,\\
    \Exp\big[\|\ddF(z;\zeta)\|_F^3\big] &\leq \sigt^3, &\quad \Exp\big[\|\ddF(z_1; \zeta) - \ddF(z_2; \zeta)\|_F^3\big] &\leq \rho^3 \|z_1 - z_2\|^3.
    \end{aligned}\]
\end{assumpt}

Note that in Assumption \ref{Assumption_1} and Assumption \ref{ass:stoch}, we use the same symbols $L$ and $\rho$ for the Lipschitz constants and their third-moment counterparts. This is mainly for notational simplicity. One can indeed assign different notations to distinguish these quantities. We now discuss how to estimate $\dhb(z)$, $\Hb(z)$, and $\dfy(z)$ required in the subproblem \eqref{eq:cubic-subproblem} via independent sampling. 
Let $\mathcal{B}_1 = \{\zeta_{i,1}\}_{i=1}^B$ and $\mathcal{B}_2 = \{\zeta_{i,2}\}_{i=1}^B$ be two independent random batches of size $B$. We define the  estimator for $\dhb(z)$ as:
\[
Q(z; \cB_1, \cB_2) := \frac{1}{B} {\sum}_{i=1}^B \Big[I_{m+n} + \beta\ddF(z; \zeta_{i,1}) P\Big] \dF(z; \zeta_{i,2}), 
\]
the estimator for $\Hb(z)$ as:
\[
W(z; \cB_1, \cB_2) := \frac{1}{B} {\sum}_{i=1}^B \Big[I_{m+n} + \beta\ddF(z; \zeta_{i,1}) P\Big] \ddF(z; \zeta_{i,2}),
\]
and the partial gradient estimator for $\dfy(z)$ as:
\[
G_y(z; \cB_1) := \frac{1}{B} {\sum}_{i=1}^B \nabla_y F(z; \zeta_{i,1}).
\]
It is straightforward to verify that the estimators $Q(z; \cB_1, \cB_2)$, $W(z; \cB_1, \cB_2)$ and $G_y(z; \cB_1)$ are unbiased estimators of $\dhb(z)$, $\Hb(z)$, and $\dfy(z)$, respectively. Since $W$ may not be symmetric, whenever it is used as the Hessian approximation in the subproblem, we replace it by $\operatorname{Sym}(W):=(W+W^\top)/2$ and still denote the symmetrized estimator by $W$. This preserves unbiasedness because $\Hb(z)$ is symmetric and $\Exp[W]=\Hb(z)$. It also preserves the deviation bounds because symmetrization does not increase the spectral or Frobenius norm error. Notably, utilizing modern automatic differentiation techniques, the computational cost for evaluating the estimator $Q(z; \cB_1, \cB_2)$ (or $W(z; \cB_1, \cB_2)$) is comparable to computing the stochastic gradient (or Hessian) batches of size $B$. 

Given the preceding discussions, we are ready to formally introduce our S-ACQRN method (Algorithm \ref{alg:stoch-cubic-minimax}). The estimator updates \eqref{eq:stochastic-estimators} for $Q$ and $W$ utilize a recursive variance reduction mechanism \citep{nguyen2017sarah}. Moreover, the design of our sampling scheme  \eqref{eq:batch-size} is motivated by \citep{zhou2020stochastic}. This approach adaptively scales the inner-epoch batch sizes in \eqref{eq:batch-size} with respect to $\|\bxikm\|^2$. Thus, when $\|\bxikm\|^2$ is small, our algorithm draws significantly fewer samples than the fixed-batch approaches. However, such a dynamic batch-size rule also creates additional difficulties in the analysis. In fact, the analysis in \citep{zhou2020stochastic} suffers from a structural conflict between a conditional high-probability truncation argument and the Azuma-style martingale-difference argument, making their final complexity claim not fully justified. To resolve this issue, we propose a new analysis strategy that relies on a different probabilistic tool: the Burkholder-Davis-Gundy (BDG) inequality \citep{BurDavGun72,burkholder1988sharp}.  


\begin{algorithm}[ht]
\caption{S-ACQRN method for \eqref{prob_stoch}.}
\label{alg:stoch-cubic-minimax}
\begin{algorithmic}[1]
\State{\textbf{Initialization:} Initial point $\boldsymbol z_0 = (\boldsymbol x_0, \boldsymbol y_0)$, tolerance $\varepsilon > 0$, function value gap upper bound $\Delta_0 >0$, total iteration count $T \in \bb{Z}_{++}$, cycle length $S^g, S^h \in \bb{Z}_{++}$, regularization parameters $\beta > \muy^{-1}$, $\alo = 2\beta\rho$, $\alt = 12c$, where $c=(3\beta L+1)\rho/6$, and parameters for the batch sizes $R_\Delta=2\sqrt{\muy\Delta_0/(\beta\muy-1)}$, $c_1=(1+\beta\sigt)^2\sigo^2+\beta^2\sigt^2R_\Delta^2$, $c_2=L^2(1+\beta\sigt)^2+\beta^2\rho^2(\sigo^2+R_\Delta^2)$, and $c_3=1+\beta(L+\sigt)$.}
\For{$k=0,1,2,\ldots,T-1$}
\State{Sample independent batches $\cBgko, \cBgkt, \cBhko, \cBhkt, \cBgky \sim Z$ with sizes satisfying
\vspace{-2mm}
\begin{equation}
\label{eq:batch-size}
\begin{aligned}
    |\cBgko|=|\cBgkt| &\geq \begin{cases}
        64c_1\varepsilon^{-2}, & \text{if } \operatorname{mod}(k, S^g) = 0;\\
        \lfloor64c_2S^g \|\bxikm\|^2 \varepsilon^{-2}\rfloor + 1, &\text{otherwise};
    \end{cases}\\  
    |\cBhko|= |\cBhkt| &\geq \begin{cases}
    648 \sigt^2 c_3^2 (c\varepsilon)^{-1}, & \text{if } \operatorname{mod}(k, S^h) = 0;\\
    \lfloor 648 \rho^2 c_3^2 S^h \|\bxikm\|^2 (c\varepsilon)^{-1}\rfloor + 1, &\text{otherwise};
    \end{cases}\\
    |\cBgky| &\geq 36 (\sigo\beta\rho)^2 (c\varepsilon)^{-1},
\end{aligned}
\end{equation}
where the constants are set in the initialization.}
\State{Compute
\begin{equation}
\label{eq:stochastic-estimators}
\begin{aligned}
    \bgk &\gets 
    \begin{cases} 
        Q(\bzk; \cBgko, \cBgkt), & \text{if } \operatorname{mod}(k, S^g) = 0; \\
        Q(\bzk; \cBgko, \cBgkt) -  Q(\bzkm; \cBgko, \cBgkt)  + \bgkm, & \text{otherwise};\\
    \end{cases}\\
    \bHbk &\gets 
    \begin{cases} 
        W(\bzk; \cBhko, \cBhkt), & \text{if } \operatorname{mod}(k, S^h) = 0; \\
        W(\bzk; \cBhko, \cBhkt) -   W(\bzkm; \cBhko, \cBhkt) + \bHbkm, & \text{otherwise};
    \end{cases}\\
    \bgyk &\gets G_y(\bzk; \cBgky).
\end{aligned}
\end{equation}}
\State{Solve the cubic regularized subproblem
\begin{equation}
    \label{eq:stoch-cubic-subproblem}
    \bxik \in \argmin_{\xi \in \bb{R}^{n+m}} \bgk^\top\xi + \frac{1}{2} \xi^\top [\bHbk + \alo \|\bgyk\| I_{n+m}] \xi + \frac{\alt}{6} \|\xi\|^3.
\end{equation}}
\vspace{-3mm}
\State{$\bzkp \gets \bzk + \bxik$.}
\EndFor
\State{\textbf{Output}: $\boldsymbol{z}_{\tilde k}$ with $\tilde k$ chosen uniformly at random from $\{1, 2, \ldots, T\}$.}
\end{algorithmic}
\end{algorithm}

\begin{prop}[Burkholder-Davis-Gundy inequality]
    \label{Thm:BDG}
		Let $\{\bwk\}_k$ be a vector-valued martingale adapted to the filtration $\{\mathcal U_k\}_k$ and $\boldsymbol{w}_{0}=0$. For all $p \in (1,\infty)$, there is a constant $C_p > 0$, depending only on $p$, such that
		\[ \Exp\Big[{{\sup}_{k \geq 0} \|\bwk\|^p}\Big] \leq C_p \cdot \Exp\Big[\Big( {\sum}_{k=1}^\infty \|\bwk-\bwkm\|^2 \Big)^{\frac{p}{2}} \Big]. \]
\end{prop}

For $p \geq 2$, it holds that $C_p = p^p$ \citep[Equation (3.4)]{burkholder1988sharp}. The BDG inequality allows us to establish the following upper bounds on the stochastic estimators $\bgk$, $\bHbk$, and $\bgyk$ under the sampling scheme \eqref{eq:batch-size}. In the following, for the estimator $\bgk$, we first prove a second-moment estimate using the BDG inequality with $p=2$. 
Throughout this section, let $\mathscr B_k$ denote the collection of all fresh random batches sampled at iteration $k$, i.e.,
\(
\mathscr B_k
:=
(\cBgko,\cBgkt,\cBhko,\cBhkt,\cBgky).
\)
Since $\boldsymbol{z}_0$ is deterministic, we define the filtration
$\cF_0:=\{\emptyset,\Omega\}$,
$\cF_k:=\sigma(\mathscr B_0,\ldots,\mathscr B_{k-1}),$ $k\ge 1$, 
and write
\(
\Exp_k[\cdot]:=\Exp[\cdot\mid \cF_k].
\)

\begin{lem}
    \label{lem:controlled-q-deviation}
    Suppose Assumptions \ref{Assumption_1} and \ref{ass:stoch} hold. Let $\{\bxik\}_k$ and $\{\bzk\}_k$ be generated by Algorithm \ref{alg:stoch-cubic-minimax}. Define $\Delta_j:=\Exp[\hb(\bzj)]-\bar\Phi \geq 0$. If $\Delta_j\leq 2\Delta_0$ for all $0\leq j\leq k$, then
    \[
    \Exp[\|\bgk-\dhb(\bzk)\|^2]\leq \varepsilon^2.
    \]
\end{lem}
\begin{proof}
By the $\mu$-strong concavity of $f(x,\cdot)$ and $\bar \Phi := \inf_x \Phi(x) > -\infty$, it holds that
\[
\hb(z)\geq \Phi(x)+\frac{\beta\muy-1}{2\muy}\|\dfy(z)\|^2
\geq \bar\Phi+\frac{\beta\muy-1}{2\muy}\|\dfy(z)\|^2.
\]
Then, for every $0\leq j\leq k$, rearranging, taking expectation, and invoking $\beta > \mu^{-1}$ yields
\[
\Exp[\|\dfy(\bzj)\|^2]\leq \frac{2\mu}{\beta\mu-1} (\Exp[\hb(\bzj)] - \bar\Phi) \leq  \frac{4\muy\Delta_0}{\beta\muy-1}=:R_\Delta^2,
\]
where the second inequality is due to $\Delta_{j} \leq 2\Delta_0$. 

Now, we prove the second-moment estimate in three major steps. First, we establish single-sample-pair bounds for the restart estimator and the recursive variance-reduction difference estimator. Second, we use conditional independence and the prescribed batch sizes to convert these bounds into mini-batch increment bounds for the martingale differences. Third, we aggregate the mini-batch increments over one epoch using the BDG inequality.

\noindent\textbf{Step 1. Single-sample-pair bounds.} For a single pair of samples, define
\[
\begin{aligned}
\boldsymbol A_{k,1}^{i}&:=I+\beta\ddF(\bzk;\zeta_{i,1})P,
&\boldsymbol b_{k,2}^{i}&:=\dF(\bzk;\zeta_{i,2}),\\
\boldsymbol A_k&:=I+\beta\ddf(\bzk)P,
&\boldsymbol b_k&:=\df(\bzk).
\end{aligned}
\]
Observe that
\(
\boldsymbol A_{k,1}^{i}\boldsymbol b_{k,2}^{i}-\boldsymbol A_k\boldsymbol b_k
=\boldsymbol A_{k,1}^{i}(\boldsymbol b_{k,2}^{i}-\boldsymbol b_k)
+(\boldsymbol A_{k,1}^{i}-\boldsymbol A_k)\boldsymbol b_k
\)
and the two terms on the RHS are orthogonal in conditional expectation ($\Ek[\boldsymbol b_{k,2}^{i}-\boldsymbol b_k]=0$ and $\zeta_{i,1},\zeta_{i,2}$ being independent). Then, it holds that
\[
\begin{aligned}
&\Ek[\|\boldsymbol A_{k,1}^{i}\boldsymbol b_{k,2}^{i}-\boldsymbol A_k\boldsymbol b_k\|^2]\\
=&\Ek[\|\boldsymbol A_{k,1}^{i}(\boldsymbol b_{k,2}^{i}-\boldsymbol b_k)\|^2] + \Ek[\|(\boldsymbol A_{k,1}^{i}-\boldsymbol A_k)\boldsymbol b_k\|^2]\\
\leq& \{\Ek[\|\boldsymbol A_{k,1}^{i}\|^3]\}^{\frac23}
\{\Ek[\|\boldsymbol b_{k,2}^{i}-\boldsymbol b_k\|^3]\}^{\frac23}+\beta^2\Ek[\|\ddF(\bzk;\zeta_{i,1})-\ddf(\bzk)\|_F^2]\|P\boldsymbol b_k\|^2\\
\leq&(1+\beta\sigt)^2\sigo^2+\beta^2\sigt^2\|\dfy(\bzk)\|^2,
\end{aligned}
\]
where the first inequality follows from Jensen's and H{\"o}lder's inequality, and in the last inequality we utilized $\|P\|=1$, $\{\Ek[\|\boldsymbol b_{k,2}^{i}-\boldsymbol b_k\|^3]\}^{\frac13}\leq \sigo$, and
\[
\begin{aligned}
\{\Ek&[\|\boldsymbol{A}_{k,1}^i\|^3]\}^{\frac13} \leq 1 + \beta \Ek[\|\ddF(\bzk;\zeta_{i,1})\|_F^3]^{\frac13}\cdot \|P\| \leq 1 + \beta \sigt,\\
\Ek&[\|\ddF(\bzk;\zeta_{i,1})-\ddf(\bzk)\|_F^2]\leq \Ek[\|\ddF(\bzk;\zeta_{i,1})\|_F^2]\leq \sigt^2.
\end{aligned}
\]
Furthermore, by Minkowski's inequality, it holds that
\[
\begin{aligned}
&\{\Ek[\|\boldsymbol A_{k,1}^{i}\boldsymbol b_{k,2}^{i}-\boldsymbol A_{k-1,1}^{i}\boldsymbol b_{k-1,2}^{i}
-\boldsymbol A_k\boldsymbol b_k+\boldsymbol A_{k-1}\boldsymbol b_{k-1}\|^2]\}^{\frac12}\\
\leq& \{\Ek[\|\boldsymbol A_{k,1}^{i}(\boldsymbol b_{k,2}^{i}-\boldsymbol b_{k-1,2}^{i})\|^2]\}^{\frac12}
+\{\Ek[\|(\boldsymbol A_{k,1}^{i}-\boldsymbol A_{k-1,1}^{i})\boldsymbol b_{k-1,2}^{i}\|^2]\}^{\frac12}\\
&\hspace{1cm} +\|\boldsymbol A_k(\boldsymbol b_k-\boldsymbol b_{k-1})\|
+\|(\boldsymbol A_k-\boldsymbol A_{k-1})\boldsymbol b_{k-1}\|\\
\leq& \{\Ek[\|\boldsymbol A_{k,1}^{i}\|^2] \cdot\Ek[\|\boldsymbol b_{k,2}^{i}-\boldsymbol b_{k-1,2}^{i}\|^2]\}^{\frac12} \\
&\hspace{1cm}
+\beta \{\Ek[\|\ddF(\bzk;\zeta_{i,1})-\ddF(\bzkm;\zeta_{i,1})\|_F^2] \cdot \Ek[\|P\boldsymbol b_{k-1,2}^{i}\|^2]\}^{\frac12}\\
&\hspace{1cm} +\|\boldsymbol A_k\|\cdot\|\boldsymbol b_k-\boldsymbol b_{k-1}\|
+\beta\|\ddf(\bzk)-\ddf(\bzkm)\|\cdot\|P\boldsymbol b_{k-1}\|\\
\leq& 2L(1+\beta\sigt)\|\bxikm\|+2\beta\rho(\sigo^2+\|\dfy(\bzkm)\|^2)^{\frac12}\|\bxikm\|,
\end{aligned}
\]
where the second inequality is due to the conditional independence of $\zeta_{i,1}$ and $\zeta_{i,2}$, and last inequality follows from the Lipschitz continuity of $\ddf$ and $\ddF$ and the following estimates:
\[
\begin{aligned}
\|\boldsymbol A_k\|&\leq 1+\beta\sigt,
&\{\Ek[\|\boldsymbol{A}_{k,1}^i\|^2]\}^{\frac12}&\leq 1+\beta\sigt,\\
\{\Ek[\|\boldsymbol b_{k,2}^{i}-\boldsymbol b_{k-1,2}^{i}\|^2]\}^{\frac12}&\leq L\|\bxikm\|,
&\|\boldsymbol b_k-\boldsymbol b_{k-1}\|&\leq L\|\bxikm\|,\\
\{\Ek[\|P\boldsymbol b_{k-1,2}^{i}\|^2]\}^{\frac12}&\leq (\sigo^2+\|\dfy(\bzkm)\|^2)^{\frac12},
&\|P\boldsymbol b_{k-1}\|&= \|\dfy(\bzkm)\|.
\end{aligned}
\]
Thus, by Young's inequality, we have
\[
\begin{aligned}
&\Ek[\|\boldsymbol A_{k,1}^{i}\boldsymbol b_{k,2}^{i}-\boldsymbol A_{k-1,1}^{i}\boldsymbol b_{k-1,2}^{i}
-\boldsymbol A_k\boldsymbol b_k+\boldsymbol A_{k-1}\boldsymbol b_{k-1}\|^2]\\
\leq& 8\|\bxikm\|^2
\big[L^2(1+\beta\sigt)^2+\beta^2\rho^2(\sigo^2+\|\dfy(\bzkm)\|^2)\big].
\end{aligned}
\]

\noindent\textbf{Step 2. Mini-batch increment bounds}
Now, observe that $\bgk-\dhb(\bzk)=\sum_{j=k_0}^k\buj$, where $k_0=\lfloor k/S^g\rfloor S^g$ and
\[
\buj=
\begin{cases}
Q(\bzj;\mathcal B^g_{j,1},\mathcal B^g_{j,2})-\dhb(\bzj), & j=k_0,\\
Q(\bzj;\mathcal B^g_{j,1},\mathcal B^g_{j,2})
-Q(\bzjm;\mathcal B^g_{j,1},\mathcal B^g_{j,2})
-\dhb(\bzj)+\dhb(\bzjm), & j>k_0.
\end{cases}
\]
When $j=k_0$, it holds that
\[
\begin{aligned}
\Exp_j[\|\buj\|^2]
&=|\mathcal B^g_{j,1}|^{-2}\Exp_j\bigg[\bigg\|{\sum}_{i=1}^{|\mathcal B^g_{j,1}|}
\boldsymbol A_{j,1}^{i}\boldsymbol b_{j,2}^{i}-\boldsymbol A_j\boldsymbol b_j\bigg\|^2\bigg]\\
&= |\mathcal B^g_{j,1}|^{-2} {\sum}_{i=1}^{|\mathcal B^g_{j,1}|} \Exp_j[\|
\boldsymbol A_{j,1}^{i}\boldsymbol b_{j,2}^{i}-\boldsymbol A_j\boldsymbol b_j\|^2]\\
&\leq 2|\mathcal B^g_{j,1}|^{-1}
\big[(1+\beta\sigt)^2\sigo^2+\beta^2\sigt^2\|\dfy(\bzj)\|^2\big],
\end{aligned}
\]
where the second line follows from the fact that the summands are conditionally independent and have zero mean.
Taking total expectation, using $\Exp[\|\dfy(\bzj)\|^2]\leq R_\Delta^2$, and invoking \eqref{eq:batch-size}, we obtain
\[
\Exp[\|\buj\|^2]\leq 2|\mathcal B^g_{j,1}|^{-1}c_1\leq \frac{\varepsilon^2}{32}.
\]
When $j>k_0$, the summands in $\buj$ are also conditionally independent and have conditional mean zero. Thus, we obtain
\[
\begin{aligned}
\Exp_j[\|\buj\|^2]
&=|\mathcal B^g_{j,1}|^{-2}\Exp_j\bigg[\bigg\|{\sum}_{i=1}^{|\mathcal B^g_{j,1}|}
\boldsymbol A_{j,1}^{i}\boldsymbol b_{j,2}^{i}-\boldsymbol A_{j-1,1}^{i}\boldsymbol b_{j-1,2}^{i}
-\boldsymbol A_j\boldsymbol b_j+\boldsymbol A_{j-1}\boldsymbol b_{j-1}\bigg\|^2\bigg]\\
&= |\mathcal B^g_{j,1}|^{-2}{\sum}_{i=1}^{|\mathcal B^g_{j,1}|} \Exp_j[\|
\boldsymbol A_{j,1}^{i}\boldsymbol b_{j,2}^{i}-\boldsymbol A_{j-1,1}^{i}\boldsymbol b_{j-1,2}^{i}
-\boldsymbol A_j\boldsymbol b_j+\boldsymbol A_{j-1}\boldsymbol b_{j-1}\|^2]\\
&\leq \frac{\varepsilon^2}{8S^g}\cdot
\frac{L^2(1+\beta\sigt)^2+\beta^2\rho^2(\sigo^2+\|\dfy(\bzjm)\|^2)}{c_2}.
\end{aligned}
\]
Taking total expectation and using $\Exp[\|\dfy(\bzjm)\|^2]\leq R_\Delta^2$ gives
\[
\Exp[\|\buj\|^2]\leq \frac{\varepsilon^2}{8S^g}.
\]

\noindent \textbf{Step 3. Epoch-level BDG aggregation.} Since the $\buj$'s are martingale differences with $\Exp_j[\buj]=0$, applying the BDG inequality  with $p=2$ and $C_2=4$ implies
\[
\begin{aligned}
\Exp[\|\bgk-\dhb(\bzk)\|^2]
= \Exp\bigg[\bigg\|{\sum}_{j=k_0}^k\buj\bigg\|^2\bigg]
&\leq C_2{\sum}_{j=k_0}^k\Exp[\|\buj\|^2]
\leq 4\bigg(\frac{\varepsilon^2}{32}+S^g\cdot\frac{\varepsilon^2}{8S^g}\bigg)
\leq \varepsilon^2.
\end{aligned}
\]
This completes the proof.
\end{proof}

Next, for $\Hbk$ and $\bgyk$, we work  with the third-moment BDG estimate using $p=3$. 

\begin{lem}
    \label{lem:controlled-whgy-deviation}
    Suppose Assumptions \ref{Assumption_1} and \ref{ass:stoch} hold. Let $\{\bxik\}_k$ and $\{\bzk\}_k$ be generated by Algorithm \ref{alg:stoch-cubic-minimax}, and set $c=(3\beta L+1)\rho/6$. Then, for all $k\geq0$, it holds that
    \[
    \Exp[\|\bHbk-\Hb(\bzk)\|_F^3]\leq (c\varepsilon)^{3/2} \qaq
    \Exp[\|\bgyk-\dfy(\bzk)\|^3]\leq(\beta\rho)^{-3}(c\varepsilon)^{3/2}.
    \]
\end{lem}
\begin{proof}
We first prove the bound for $\bHbk$ in three major steps similar to those of the previous lemma: (1) establish single-sample-pair bounds for the restart estimator and the recursive variance-reduction difference estimator; (2) convert these bounds into mini-batch increment bounds for the martingale differences; (3) aggregate the mini-batch increments over one epoch using the BDG inequality.

\noindent\textbf{Step 1. Single-sample-pair bounds.} For a single pair of Hessian samples, define
\[
\begin{aligned}
\boldsymbol A_{k,1}^{i}&:=I_{n+m}+\beta\ddF(\bzk;\zeta_{i,1})P,
&\boldsymbol B_{k,2}^{i}&:=\ddF(\bzk;\zeta_{i,2}),\\
\boldsymbol A_k&:=I_{n+m}+\beta\ddf(\bzk)P,
&\boldsymbol B_k&:=\ddf(\bzk).
\end{aligned}
\]
Invoking  Minkowski's inequality and the conditional independence of $\cBhko$ and $\cBhkt$ yields
\[\begin{aligned}
    &\{\Ek[\|\boldsymbol{A}_{k,1}^{i} \boldsymbol{B}_{k,2}^{i} - \boldsymbol{A}_{k}\boldsymbol{B}_{k} \|_F^3]\}^{\frac13} \\
    \leq& \{\Ek[\|\boldsymbol{A}_{k,1}^{i} - \boldsymbol{A}_{k}\|^3] \cdot \Ek[\|\boldsymbol{B}_{k,2}^{i} \|_F^3]\}^{\frac13}
    + \|\boldsymbol{A}_{k}\| \{ \Ek[\|\boldsymbol{B}_{k,2}^{i} - \boldsymbol{B}_{k} \|_F^3]\}^{\frac13}
    \\
    \leq& 2\beta\sigt^2 + 2(1+\beta L)\sigt = 2[1+\beta(\sigt + L)]\sigt=
    2c_3\sigt,
\end{aligned}\]
where in the second inequality we utilize the following estimates:
\[\begin{aligned}
    \|\boldsymbol{A}_k\| &\leq 
    1 + \beta \cdot \|\ddf(\bzk)\| \cdot \|P\| \leq 1 + \beta L,\\
    \{\Ek[\|\boldsymbol{A}_{k,1}^{i} - \boldsymbol{A}_{k}\|^3]\}^{\frac13} &\leq \beta \{\Ek[\|\ddF(\bzk;\zeta_{i,1}) - \ddf(\bzk)\|_F^3]\}^{\frac13} \cdot \|P\| 
    \leq 2\beta\sigt,\\
    \{ \Ek[\|\boldsymbol{B}_{k,2}^{i}\|_F^3] \}^{\frac13} &\leq \sigt, \qaq \{\Ek[\|\boldsymbol{B}_{k,2}^{i} - \boldsymbol{B}_{k}\|_F^3]\}^{\frac13} \leq 2\sigt.
\end{aligned}\]
Furthermore, it holds that
\[\begin{aligned}
    &\{\Ek[\|\boldsymbol{A}_{k,1}^{i} \boldsymbol{B}_{k,2}^{i} - \boldsymbol{A}_{k-1,1}^{i} \boldsymbol{B}_{k-1,2}^{i}- \boldsymbol{A}_{k}\boldsymbol{B}_{k} + \boldsymbol{A}_{k-1}\boldsymbol{B}_{k-1} \|_F^3]\}^{\frac13} \\
    \leq& \{\Ek[\|\boldsymbol{A}_{k,1}^{i} - \boldsymbol{A}_{k-1,1}^{i}\|^3] \cdot \Ek[\|\boldsymbol{B}_{k,2}^{i} \|_F^3]\}^{\frac13} + 
    \{\Ek[\|\boldsymbol{A}_{k-1,1}^{i}\|^3] \cdot \Ek[\|\boldsymbol{B}_{k,2}^{i} - \boldsymbol{B}_{k-1,2}^{i}\|_F^3]\}^{\frac13} + \\
    &\hspace{1cm} \|\boldsymbol{A}_{k} - \boldsymbol{A}_{k-1}\|_F \cdot \|\boldsymbol{B}_{k} \| + \|\boldsymbol{A}_{k-1}\| \cdot \|\boldsymbol{B}_{k} - \boldsymbol{B}_{k-1}\|_F
    \\
    \leq& [\beta\rho\cdot \sigt + (1+\beta\sigt)\cdot \rho + \beta\rho\cdot L + (1+\beta L)\cdot \rho]\|\bxikm\|\\
    \leq& 2\rho[1 + \beta(\sigt + L)]\|\bxikm\| = 2\rho c_3\|\bxikm\|,
\end{aligned}\]
where in the second inequality we utilize the following additional estimates:
\[
\begin{aligned} 
    \Ek[\|\boldsymbol{A}_{k-1,1}^i\|^3]^{\frac13} &\leq 1 + \beta \Ek[\|\ddF(\bzkm;\zeta_{i,1})\|_F^3]^{\frac13}\cdot \|P\| \leq 1 + \beta \sigt,\\
    \|\boldsymbol{A}_k - \boldsymbol{A}_{k-1}\|_F &\leq \beta \|\ddf(\bzk) - \ddf(\bzkm)\|_F \cdot \|P\| \leq \beta \rho \|\bxikm\|,\\
    \{\Ek[\|\boldsymbol{A}_{k,1}^{i} - \boldsymbol{A}_{k-1,1}^{i}\|^3]\}^{\frac13} &\leq \beta \{\Ek[\|\ddF(\bzk;\zeta_{i,1}) - \ddF(\bzkm;\zeta_{i,1})\|_F^3]\}^{\frac13} \|P\|
    \leq \beta \rho \|\bxikm\|,\\
    \|\boldsymbol{B}_k - \boldsymbol{B}_{k-1}\|_F &\leq \rho \|\bxikm\|,\qaq
    \{\Ek[\|\boldsymbol{B}_{k,2}^{i} - \boldsymbol{B}_{k-1,2}^{i}\|_F^3]\}^{\frac13} \leq \rho \|\bxikm\|.
\end{aligned}
\]
Here the deterministic Frobenius-norm $\rho$-Lipschitz continuity of $\ddf$ follows from the unbiasedness of $\ddF$ and Jensen's inequality applied to the mean cubic Hessian smoothness in Assumption \ref{ass:stoch}.

\noindent\textbf{Step 2. Mini-batch increment bounds.} Now, observe that $\bHbk - \Hb(\bzk) = \sum_{j=k_0}^k \buj,$ where $k_0 = \lfloor \frac{k}{S^h}\rfloor \cdot S^h$ and 
\[
\buj=
\begin{cases}
W(\bzj;\mathcal B^h_{j,1},\mathcal B^h_{j,2})-\Hb(\bzj), & j=k_0,\\
W(\bzj;\mathcal B^h_{j,1},\mathcal B^h_{j,2})
-W(\bzjm;\mathcal B^h_{j,1},\mathcal B^h_{j,2})
-\Hb(\bzj)+\Hb(\bzjm), & j>k_0.
\end{cases}
\]
Since the two batches $\mathcal B^h_{j,1}$ and $\mathcal B^h_{j,2}$ are jointly independent (conditioning on $\cF_j$), by the BDG inequality (Proposition \ref{Thm:BDG}) with $p = 3$,
it holds that when $j = k_0$,
\[\begin{aligned}
\Exp_j[\|\buj\|_F^{3}]
=& |\mathcal B^h_{j,1}|^{-3} \Exp_j \bigg[ \bigg\| {\sum}_{i=1}^{|\mathcal B^h_{j,1}|}  \boldsymbol{A}_{j,1}^{i} \boldsymbol{B}_{j,2}^{i} - \boldsymbol{A}_{j}\boldsymbol{B}_{j}\bigg\|_F^{3} \bigg]\\
\leq& C_{3} |\mathcal B^h_{j,1}|^{-3}  \Exp_j \bigg[ \bigg({\sum}_{i=1}^{|\mathcal B^h_{j,1}|} \|\boldsymbol{A}_{j,1}^{i} \boldsymbol{B}_{j,2}^{i} - \boldsymbol{A}_{j}\boldsymbol{B}_{j}\|_F^2\bigg)^{\frac32} \bigg]\\
\leq& C_3 |\mathcal B^h_{j,1}|^{-3} \bigg[{\sum}_{i=1}^{|\mathcal B^h_{j,1}|}   \bigg( \Exp_j[\|\boldsymbol{A}_{j,1}^{i} \boldsymbol{B}_{j,2}^{i} - \boldsymbol{A}_{j}\boldsymbol{B}_{j}\|_F^3]\bigg)^{\frac23}\bigg]^{\frac32} \\
\leq& 8C_3c_3^3\sigt^3|\mathcal B^h_{j,1}|^{-\frac32}\leq C_3(162)^{-\frac32}(c\varepsilon)^{\frac32},
\end{aligned}\]
where the second inequality follows from Minkowski's inequality for $\cL^{1.5}$-norm
and the last inequality follows from 
\eqref{eq:batch-size}.
When $j > k_0$, it holds that
\[\begin{aligned}
\Exp_j[\|\buj\|_F^{3}] 
=& |\mathcal B^h_{j,1}|^{-3}  \Exp_j \bigg[\bigg\|  {\sum}_{i=1}^{|\mathcal B^h_{j,1}|}  \boldsymbol{A}_{j,1}^{i} \boldsymbol{B}_{j,2}^{i} - \boldsymbol{A}_{j-1,1}^{i} \boldsymbol{B}_{j-1,2}^{i}- \boldsymbol{A}_{j}\boldsymbol{B}_{j} + \boldsymbol{A}_{j-1}\boldsymbol{B}_{j-1}\bigg\|_F^3 \bigg]\\
\leq& C_{3} |\mathcal B^h_{j,1}|^{-3}  \Exp_j \bigg[\bigg(  {\sum}_{i=1}^{|\mathcal B^h_{j,1}|}  \|\boldsymbol{A}_{j,1}^{i} \boldsymbol{B}_{j,2}^{i} - \boldsymbol{A}_{j-1,1}^{i} \boldsymbol{B}_{j-1,2}^{i}- \boldsymbol{A}_{j}\boldsymbol{B}_{j} + \boldsymbol{A}_{j-1}\boldsymbol{B}_{j-1}\|_F^2\bigg)^{\frac32} \bigg]\\
\leq& C_{3} |\mathcal B^h_{j,1}|^{-3} \bigg[{\sum}_{i=1}^{|\mathcal B^h_{j,1}|} \bigg(\Exp_j [ \|\boldsymbol{A}_{j,1}^{i} \boldsymbol{B}_{j,2}^{i} - \boldsymbol{A}_{j-1,1}^{i} \boldsymbol{B}_{j-1,2}^{i}- \boldsymbol{A}_{j}\boldsymbol{B}_{j} + \boldsymbol{A}_{j-1}\boldsymbol{B}_{j-1}\|_F^3] \bigg)^{\frac23} \bigg]^{\frac32}\\
\leq& 8C_{3}\rho^3 c_3^3 \|\boldsymbol{\xi}_{j-1}\|^3|\mathcal B^h_{j,1}|^{-\frac32} \leq C_3(162S^h)^{-\frac32} (c\varepsilon)^{\frac32},
\end{aligned}\]
where the last inequality again follows from 
\eqref{eq:batch-size}. 

\noindent \textbf{Step 3. Epoch-level BDG aggregation.} Since the $\buj$'s are martingale differences with $\Exp_j[\buj] = 0$, 
the BDG inequality implies
\[\begin{aligned}
\Exp[\|\bHbk - \Hb(\bzk)\|_F^{3}] &\leq C_{3}  \Exp \bigg[  \bigg({\sum}_{j=k_0}^k \|\buj\|_F^2\bigg)^{\frac32} \bigg]
\leq C_{3}  \bigg( {\sum}_{j=k_0}^k \Big(\Exp [ \|\buj\|_F^3 ]\Big)^{\frac23}\bigg)^{\frac32}\\
&\leq C_{3}  \Big[C_{3}^{\frac23}(162)^{-1}c\varepsilon + S^h \cdot C_3^{\frac23}(162S^h)^{-1} c\varepsilon\Big]^{\frac32}= 81^{-\frac32}C_3^2 (c\varepsilon)^{\frac32} = (c\varepsilon)^{\frac32},
\end{aligned}\]
where the third inequality follows from the previous two estimates and the fact that $k - k_0 \leq S^h$, and the last inequality follows from the fact that $C_3^2 = 27^2 = 81^{3/2}$. This completes the proof of the bound for $\bHbk$. 

For the bound on $\bgyk$, we directly apply the BDG inequality:
\[\begin{aligned}
    \Exp[\|\bgyk - \dfy(\bzk)\|^{3}] &\leq C_{3}  \Exp \bigg[ \bigg({\sum}_{\zeta \in \cBgky} \frac{1}{|\cBgky|^2} \|\nabla_y F(\bzk; \zeta ) - \dfy(\bzk)\|^2\bigg)^{\frac32} \bigg]\\
    &\leq C_3 \bigg\{{\sum}_{\zeta \in \cBgky}  \bigg[\frac{1}{|\cBgky|^3} \Exp[\|\nabla_y F(\bzk; \zeta ) - \dfy(\bzk)\|^3]\bigg]^{\frac23}\bigg\}^{\frac32} \\
    &\leq 8C_3  \sigo^3 |\cBgky|^{-3/2} \leq (\beta \rho)^{-3} (c\varepsilon)^{\frac32},
\end{aligned}\]
where the last inequality follows from \eqref{eq:batch-size} and the fact that $C_3 = 27$.
\end{proof}

These key lemmas allow us to establish the following iteration complexity for our S-ACQRN method (Algorithm \ref{alg:stoch-cubic-minimax}).

\begin{theo}
    \label{thm:convergence-rates-stoch}
        Suppose Assumptions \ref{Assumption_1} and \ref{ass:stoch} hold, let $\Delta_0 \geq \hb(\boldsymbol{z}_0) - \bar \Phi$, and set $c=(3\beta L+1)\rho/6$, $\beta>\mu^{-1}$. Then, for sufficiently small $\varepsilon>0$, running Algorithm \ref{alg:stoch-cubic-minimax} for 
        \[T = \bigg\lfloor \frac{\Delta_0 c^{\frac12}}{6\varepsilon^{\frac32}}\bigg\rfloor + 1\] iterations, the output $z_{\text{out}}$ will satisfy
    \[
        \Exp[\|\dhb(z_{\text{out}})\|]= O(\varepsilon) \qaq -\Exp[\lambda_{\min}(\nabla^2\hb(z_{\text{out}}))] \leq O(\sqrt{c\varepsilon}).
    \]
\end{theo}
\begin{proof}
Let $\Delta_k:=\Exp[\hb(\bzk)]-\bar\Phi$. We first prove by induction that $\Delta_k\leq 2\Delta_0$ for all $0\leq k\leq T$. The base case holds trivially. Suppose this holds up to iteration $k$. Then Lemma \ref{lem:controlled-q-deviation} is applicable. With $\alo=2\beta\rho$ and $\alt=2(3\beta L+1)\rho=12c$, Lemma \ref{lem:func-value-descent} gives
\[
\begin{aligned}
c\|\bxik\|^3+\frac{\beta\rho}{2}\|\dfy(\bzk)\|\|\bxik\|^2
&\leq \hb(\bzk)-\hb(\bzkp)
+\frac{2}{3\sqrt c}\|\dhb(\bzk)-\bgk\|^{3/2}\\
&\quad+\frac{1}{6c^2}\|\Hb(\bzk)-\bHbk\|^3
+\frac{4(\beta\rho)^3}{3c^2}\|\dfy(\bzk)-\bgyk\|^3.
\end{aligned}
\]
By Lemmas \ref{lem:controlled-q-deviation}, \ref{lem:controlled-whgy-deviation}, and the monotonicity of $L_p$ norm in $p$, stochastic error terms on the RHS of this inequality are bounded:
\[
\begin{aligned}
&\frac{2}{3\sqrt c}\Exp[\|\dhb(\bzk)-\bgk\|^{3/2}]
+\frac{1}{6c^2}\Exp[\|\Hb(\bzk)-\bHbk\|^3]
+\frac{4(\beta\rho)^3}{3c^2}\Exp[\|\dfy(\bzk)-\bgyk\|^3]\\
\leq&
\frac{2}{3}c^{-1/2}\varepsilon^{3/2}
+\frac{1}{6}c^{-1/2}\varepsilon^{3/2}
+\frac{4}{3}c^{-1/2}\varepsilon^{3/2}
\leq 3c^{-1/2}\varepsilon^{3/2}.
\end{aligned}
\]
Thus, taking expectation over the first inequality yields
\[\begin{aligned}
c\Exp[\|\bxik\|^3]+\frac{\beta\rho}{2}\Exp[\|\dfy(\bzk)\|\|\bxik\|^2]
&\leq \Delta_k-\Delta_{k+1}+3c^{-1/2}\varepsilon^{3/2}.
\end{aligned}\]
This implies that $\Delta_{k+1}\leq \Delta_k+3c^{-1/2}\varepsilon^{3/2}$. For $\varepsilon$ sufficiently small, the choice of $T$ guarantees 
\[\Delta_{k+1} \leq \Delta_0 + 3T c^{-1/2}\varepsilon^{3/2} \leq 2\Delta_0,\] 
which closes the induction. 

Next, we bound $\Exp[\|\dhb(z_{\text{out}})\|]$ and $-\Exp[\lambda_{\min}(\nabla^2\hb(z_{\text{out}}))]$. Telescoping our previous estimate from $k=0$ to $T-1$ and dividing by $T$ gives
\begin{equation}
    \label{eq:cubic-stoch-compl-inter0}
    \frac{1}{T}{\sum}_{k=0}^{T-1} \Big\{c\Exp[\|\bxik\|^3] + \frac{\beta\rho}2 \Exp[\|\dfy(\bzk)\|\|\bxik\|^2]\Big\}
    \leq \frac{\Delta_0}{T} + 3c^{-1/2}\varepsilon^{3/2}
    =
    O(c^{-1/2}\varepsilon^{3/2}).
\end{equation}
Let $k^* = \tilde k -1$, where $\tilde k$ is the index of the output. Then, \eqref{eq:cubic-stoch-compl-inter0} implies
\[
c\Exp[\|\bxis\|^3] + \beta\rho\Exp[\|\dfy(\bzs)\|\|\bxis\|^2] = O(c^{-1/2}\varepsilon^{3/2}).
\]
Since $k^*$ is sampled uniformly after the iterates are generated, and the induction has verified $\Delta_j\leq 2\Delta_0$ for all $j\leq T$, the estimator bounds at $k^*$ follow by averaging the deterministic-index bounds in Lemmas \ref{lem:controlled-q-deviation} and \ref{lem:controlled-whgy-deviation}. Furthermore, in the following, we repeatedly use the fact that $c = \Omega(\beta L \rho)$ and that the $L_p$ norm is nondecreasing in $p$.
Now, invoking Lemma \ref{lem:relationship-dfy-xik}, multiplying by $\|\bxis\|$, and taking expectation yields
\[\begin{aligned}
\Exp[\|\dfy(\bzs)\|\|\bxis\|]
&\leq L\Exp[\|\bxis\|^2]+(\beta\muy-1)^{-1}\Big[3\beta\rho\Exp[\|\dfy(\bzs)\|\|\bxis\|^2]+10c\Exp[\|\bxis\|^3]\\
&\quad+\Exp[\|\dhb(\bzs)-\boldsymbol g_{k^*}\|\|\bxis\|]
+(2c)^{-1}\Exp[\|\Hb(\bzs)-\bar{\boldsymbol H}_{k^*}\|^2\|\bxis\|]\\
&\quad+2(\beta\rho)^2c^{-1}\Exp[\|\dfy(\bzs)-\boldsymbol g_{y,k^*}\|^2\|\bxis\|]\Big]\\
&=O\big(Lc^{-1}\varepsilon+(\beta\muy-1)^{-1}c^{-1/2}\varepsilon^{3/2}\big),
\end{aligned}\]
where the last line follows from H{\"o}lder's inequality, Jensen's inequality, and Lemmas \ref{lem:controlled-q-deviation} and \ref{lem:controlled-whgy-deviation}:
\[
\begin{aligned}
\Exp[\|\dhb(\bzs)-\boldsymbol g_{k^*}\|\|\bxis\|]
&\leq \{\Exp[\|\dhb(\bzs)-\boldsymbol g_{k^*}\|^{3/2}]\}^{2/3}\cdot\{\Exp[\|\bxis\|^3]\}^{1/3}\\
&= O(c^{-1/2}\varepsilon^{3/2}),\\
(2c)^{-1}\Exp[\|\Hb(\bzs)-\bar{\boldsymbol H}_{k^*}\|^2\|\bxis\|]
&\leq (2c)^{-1}\{\Exp[\|\Hb(\bzs)-\bar{\boldsymbol H}_{k^*}\|^3]\}^{2/3}\cdot\{\Exp[\|\bxis\|^3]\}^{1/3}\\
&= O(c^{-1/2}\varepsilon^{3/2}),\\
2(\beta\rho)^2c^{-1}
\Exp[\|\dfy(\bzs)-\boldsymbol g_{y,k^*}\|^2\|\bxis\|]
&\leq 2(\beta\rho)^2c^{-1}
\{\Exp[\|\dfy(\bzs)-\boldsymbol g_{y,k^*}\|^3]\}^{2/3}\cdot\{\Exp[\|\bxis\|^3]\}^{1/3}\\
&= O(c^{-1/2}\varepsilon^{3/2}).
\end{aligned}
\]
Combining this inequality with Lemma \ref{lem:upper-bound-of-h-gradient} gives
\[\begin{aligned}
\Exp[\|\nabla \hb(\bzsp)\|]
&\leq 3\beta\rho\Exp[\|\dfy(\bzs)\|\|\bxis\|]+10c\Exp[\|\bxis\|^2]+\Exp[\|\dhb(\bzs)-\boldsymbol g_{k^*}\|]\\
&\quad+(2c)^{-1}\Exp[\|\Hb(\bzs)-\bar{\boldsymbol H}_{k^*}\|^2]
+2(\beta\rho)^2c^{-1}\Exp[\|\dfy(\bzs)-\boldsymbol g_{y,k^*}\|^2]\\
&=O\big(\varepsilon+(\beta\muy-1)^{-1}c^{-1/2}\beta\rho\varepsilon^{3/2}\big),
\end{aligned}\]
where the error terms are controlled by
\[
\begin{aligned}
\Exp[\|\dhb(\bzs)-\boldsymbol g_{k^*}\|]&\leq \varepsilon,\\
(2c)^{-1}\Exp[\|\Hb(\bzs)-\bar{\boldsymbol H}_{k^*}\|^2]
&\leq (2c)^{-1}\{\Exp[\|\Hb(\bzs)-\bar{\boldsymbol H}_{k^*}\|^3]\}^{2/3}
= O(\varepsilon),\\
2(\beta\rho)^2c^{-1}\Exp[\|\dfy(\bzs)-\boldsymbol g_{y,k^*}\|^2]
&\leq 2(\beta\rho)^2c^{-1}\{\Exp[\|\dfy(\bzs)-\boldsymbol g_{y,k^*}\|^3]\}^{2/3}
= O(\varepsilon).
\end{aligned}
\]
Applying Lemma \ref{lem:relationship-dfy-xik} once more and taking expectation yields
\[\begin{aligned}
\Exp[\|\dfy(\bzs)\|]
&\leq L\Exp[\|\bxis\|]+(\beta\muy-1)^{-1}\Big[3\beta\rho\Exp[\|\dfy(\bzs)\|\|\bxis\|]+10c\Exp[\|\bxis\|^2]\\
&\quad+\Exp[\|\dhb(\bzs)-\boldsymbol g_{k^*}\|]+(2c)^{-1}\Exp[\|\Hb(\bzs)-\bar{\boldsymbol H}_{k^*}\|^2]\\
&\quad+2(\beta\rho)^2c^{-1}\Exp[\|\dfy(\bzs)-\boldsymbol g_{y,k^*}\|^2]\Big]\\
&=O\big(Lc^{-1/2}\varepsilon^{1/2}+(\beta\muy-1)^{-1}\varepsilon+(\beta\muy-1)^{-2}c^{-1/2}\beta\rho\varepsilon^{3/2}\big).
\end{aligned}\]
Combining this inequality with Lemma \ref{lem:lower-bound-surrogate-hessian} gives
\[\begin{aligned}
-\Exp[\lambda_{\min}(\nabla^2 \hb(\bzsp))]
&\leq 12c\Exp[\|\bxis\|]+3\beta\rho\Exp[\|\dfy(\bzs)\|]
+2\beta\rho\Exp[\|\dfy(\bzs)-\boldsymbol g_{y,k^*}\|]\\
&\quad+\Exp[\|\Hb(\bzs)-\bar{\boldsymbol H}_{k^*}\|]\\
&=O\big((c\varepsilon)^{1/2}+(\beta\muy-1)^{-1}\beta\rho\varepsilon+(\beta\muy-1)^{-2}c^{-1/2}(\beta\rho)^2\varepsilon^{3/2}\big),
\end{aligned}\]
where the error terms are controlled by
\[
\Exp[\|\dfy(\bzs)-\boldsymbol g_{y,k^*}\|]\leq(\beta\rho)^{-1}(c\varepsilon)^{1/2}
\qaq
\Exp[\|\Hb(\bzs)-\bar{\boldsymbol H}_{k^*}\|]\leq(c\varepsilon)^{1/2}.
\]
This completes the proof.
\end{proof}

This theorem immediately yields the following complexity result.

\begin{coro}
    \label{coro:complexity-stoch}
    Suppose Assumptions \ref{Assumption_1} and \ref{ass:stoch} hold, let $\Delta_0 \geq \Exp[\hb(\boldsymbol{z}_0)]-\bar\Phi$, and set $\beta = 2\muy^{-1}$, $c=(3\beta L+1)\rho/6$, $R_\Delta=2\sqrt{\muy\Delta_0/(\beta\muy-1)}$, $c_1=(1+\beta\sigt)^2\sigo^2+\beta^2\sigt^2R_\Delta^2$, $c_2=L^2(1+\beta\sigt)^2+\beta^2\rho^2(\sigo^2+R_\Delta^2)$, and $c_3=1+\beta(L+\sigt)$. Then, for $\varepsilon > 0$ sufficiently small, Algorithm \ref{alg:stoch-cubic-minimax} with cycle length
    \[
    S^g = \bigg\lfloor \frac{\sqrt{c_1c}}{\sqrt{c_2\varepsilon}} \bigg\rfloor + 1 
    \qaq S^h = \bigg\lfloor \frac{\sigt \sqrt{c}}{\rho\sqrt{\varepsilon}} \bigg\rfloor + 1,
    \]
    finds an $(\varepsilon, \sqrt{\kappa\rho\varepsilon})$-SOSP of \eqref{Prob_Min} (in expectation) with
    \[
    T=O(\sqrt{\kappa \rho}\Delta_0\varepsilon^{-3/2})
    \]
    iterations (subproblem oracle calls),
    \[
    O\Big(\sqrt{c_1c_2}\Delta_0\varepsilon^{-3}
    +\sigo^2\beta^2\rho^2c^{-1/2}\Delta_0\varepsilon^{-5/2}\Big)
    \]
    stochastic gradient evaluations and stochastic Hessian-vector products, and
    \[
    O\big(\sigt\rho c_3^2c^{-1}\Delta_0\varepsilon^{-2}\big)
    \]
    stochastic Hessian evaluations.
\end{coro}
\begin{proof}
    From the previous theorem, $T=O(c^{1/2}\Delta_0\varepsilon^{-3/2})$. By H{\"o}lder's inequality and Jensen's inequality,
    \[
    \begin{aligned}
        {\sum}_{k=0}^{T-1} \Exp[\|\bxik\|^2]
        \leq T^{\frac13} \Big[{\sum}_{k=0}^{T-1} (\Exp[\|\bxik\|^2])^{\frac32}\Big]^{\frac23}
        \leq T^{\frac13} \Big[{\sum}_{k=0}^{T-1} \Exp[\|\bxik\|^3]\Big]^{\frac23} = O(\Delta_0(c\varepsilon)^{-1/2}).
    \end{aligned}
    \]
    Since the paired batches in $Q$ and $W$ have equal sizes, counting one representative batch changes the oracle complexities only by an absolute constant factor. Therefore,
    \[
    \begin{aligned}
        {\sum}_{k=0}^{T-1} \Exp[|\cBgko|]
        &= {\sum}_{k:\operatorname{mod}(k, S^g)=0} \Exp[|\cBgko|] + {\sum}_{k:\operatorname{mod}(k, S^g)\neq0} \Exp[|\cBgko|]\\
        &\leq O\Big(\varepsilon^{-2}\big[c_1T(S^g)^{-1}+c_2S^g{\sum}_{k=0}^{T-1}\Exp[\|\bxik\|^2]\big]\Big)\\
        &=O(\sqrt{c_1c_2}\Delta_0\varepsilon^{-3}).
    \end{aligned}
    \]
    Similarly,
    \[
    \begin{aligned}
        {\sum}_{k=0}^{T-1} \Exp[|\cBhko|]
        &= {\sum}_{k:\operatorname{mod}(k, S^h)=0} \Exp[|\cBhko|] + {\sum}_{k:\operatorname{mod}(k, S^h)\neq0} \Exp[|\cBhko|]\\
        &\leq O\Big(c_3^2(c\varepsilon)^{-1}\big[T(S^h)^{-1}\sigt^2+S^h\rho^2{\sum}_{k=0}^{T-1}\Exp[\|\bxik\|^2]\big]\Big)\\
        &=O\big(\sigt\rho c_3^2c^{-1}\Delta_0\varepsilon^{-2}\big).
    \end{aligned}
    \]
    Lastly,
    \[
    \begin{aligned}
        {\sum}_{k=0}^{T-1} \Exp[|\cBgky|]
        &=O\big(T\sigo^2\beta^2\rho^2(c\varepsilon)^{-1}\big)
        =O\big(\sigo^2\beta^2\rho^2c^{-1/2}\Delta_0\varepsilon^{-5/2}\big).
    \end{aligned}
    \]
    The final $\varepsilon$-orders follow from $\beta=2/\muy$ and $c=O(\kappa\rho)$. 
    This completes the proof.
\end{proof}

The complexity results established in Theorem \ref{thm:convergence-rates-stoch} and Corollary \ref{coro:complexity-stoch} provide several significant improvements over the existing literature. First, compared to the double-loop stochastic CRN method for NC-SC minimax problems \citep{chen2023a}, our single-loop framework achieves strictly better sample complexities in terms of the target accuracy $\varepsilon$. Second, our theoretical guarantees are established under strictly weaker assumptions regarding the stochastic oracle. While prior stochastic CRN methods \citep{zhou2020stochastic,chen2023a}  assume almost sure boundedness and Lipschitz smoothness of the stochastic gradient and Hessian to facilitate high-probability concentration inequalities, our analysis only requires a bounded third central moment for the stochastic gradient, a bounded third moment for the stochastic Hessian, and mean cubic smoothness. In particular, the bounded stochastic gradient condition is incompatible with the strong concavity assumption for NC-SC problems. Third,  when reducing our formulation to standard stochastic nonconvex minimization by setting $\beta = 0$, our algorithm naturally serves as a variance-reduced stochastic CRN method. We provide a new analysis framework based on BDG inequality to bypass the problematic high-probability arguments that cause technical gaps in the existing works \citep{zhou2020stochastic}.   

\section{Numerical Experiments}
\label{sec:numerical experiments}
In this section, we present numerical experiments on a robust nonlinear regression problem to validate the performance of our proposed deterministic method (Algorithm~\ref{alg:cubic-minimax}) and its stochastic counterpart (Algorithm~\ref{alg:stoch-cubic-minimax}). Specifically, given a feature vector $\boldsymbol w\in\mathbb R^d$ and a response $\boldsymbol v \in \bb R$ drawn from a joint distribution $(\boldsymbol w, \boldsymbol v) \sim Z$, we consider the following NC-SC minimax formulation:
\begin{equation}
\label{eq:robust-regression-minimax}
    \min_{x \in \mathbb{R}^d}\;
    \max_{y \in \mathbb{R}^{d'}}
    f(x,y)
    :=
    \Exp_{(\boldsymbol w, \boldsymbol v) \sim Z}\bigg[ \phi\Big(\langle \boldsymbol  w , x\rangle - \boldsymbol  v - p(y;\boldsymbol w, \boldsymbol v)\Big) + \frac{\rho_x}{2} \|x\|^2
        - \frac{\rho_y}{2}\|y\|^2 \bigg],
\end{equation}
where $\phi(\theta) = \theta^2/(1+\theta^2)$ is the smooth biweight loss function 
\citep{BeatonTukey1974,CarmonDuchiHinderSidford2017,LeeWright2020}, and $\rho_x\|x\|^2/2$ is a weight-decay regularization term. To promote robustness against outliers and corrupted features, we introduce an adversarial perturbation $p(y; \boldsymbol w , \boldsymbol v)$ parameterized by $y$. For simplicity, we consider a linear perturbation $p(y;\boldsymbol w,v):=[\boldsymbol w^\top,v]y$ (so $d' = d+1$) and regularize its magnitude through a quadratic penalty on $y$. Similar adversarially perturbed regression models have been widely used in robust learning; see, e.g., \citep{sinha2017certifying,lin2020gradient,ribeiro2023regularization,wang2024efficient}. All numerical experiments are implemented on a server equipped with an Apple M1 Max CPU, running Python 3.13.5, NumPy 2.1.3, and SciPy 1.15.3. 

We utilize the E2006-tfidf dataset from the LIBSVM library\footnote{Available at \url{https://www.csie.ntu.edu.tw/~cjlin/libsvmtools/datasets/regression.html}.}, which contains 16,087 data points and 150,360 features. We apply a sparse random projection\footnote{Implemented using \texttt{sklearn.random\_projection.SparseRandomProjection} with epsilon set to 0.3.} to this dataset, reducing the feature dimension to 1,076. Moreover, we preprocess the data by mean-centering both the features and response to zero.
%
The following subsections describe the deterministic and stochastic experimental settings and report the corresponding results.

\subsection{Deterministic Setting}
With a concrete dataset, the expectation in the robust regression \eqref{eq:robust-regression-minimax} reduces to an empirical average, allowing the problem to be solved via deterministic first- and second-order methods. We fix $\rho_x=0.01$ and vary $\rho_y$ to generate instances with different condition numbers $\kappa$. The initial points are set to all-zero vectors and the termination criterion is $\|\nabla f\| \le 10^{-12}$. For the CRN-based methods, the maximum iteration is set to 100 and all cubic subproblems are solved using a Lanczos solver~\citep{cartis2011adaptivea,cartis2011adaptiveb}. The tested first-order methods are run for at most 2000 iterations. 
In the following, we present the details of all tested algorithms. 

\noindent \textbf{ACQRN}: our single-loop ACQRN method (Algorithm \ref{alg:cubic-minimax}). We set $\beta = 2\mu^{-1}$, $\alo = 2\beta\rho$, and $\alt = 2(3\beta L + 1)\rho$, which match our theoretical analysis.  \vspace{0.1cm}

\noindent \textbf{D-CRN}: the double-loop CRN method \citep[Algorithm 2]{luo2022finding}. All hyperparameters are selected according to the theoretical results in \citep[Section 3]{luo2022finding}.\vspace{0.1cm}

\noindent \textbf{V-CRN}: vanilla CRN method \citep{nesterov2006cubic} applied to the minimization problem \eqref{Prob_Min} with $\beta = 2\mu^{-1}$.\vspace{0.1cm}

\noindent \textbf{TTGDA}: a modern variant of gradient-descent-ascent (GDA) method \citep{lin2020gradient,lin2025two}. The step sizes $\eta_y$ and $\eta_x$ are chosen from the best combinations in $\eta_y \in \{0.001, 0.005,$ $ 0.01, 0.05, 0.1\}$ and $\eta_x = \theta \eta_y$ with $\theta \in \{0.001, 0.01, 0.1\}$. It is worth noting that the step sizes with convergence guarantees in \cite[Theorem 17]{lin2025two}, are much smaller and lead to slow convergence. \vspace{0.1cm}

\noindent \textbf{GDA-BB}: GDA method with nonmonotone line search and adaptive BB step sizes \citep{ma2026linesearch}. The hyperparameters are selected according to  \citep[Section 5]{ma2026linesearch}. \vspace{0.1cm}
    
\noindent \textbf{L-BFGS-B}: The "L-BFGS-B" solver in the SciPy.minimize method applied to the minimization problem \eqref{Prob_Min} with $\beta = 2\mu^{-1}$.\vspace{0.1cm}

%
\noindent\textbf{Results for the deterministic setting.}
The numerical results are reported in Figure~\ref{fig:cubic-deter}. When $\kappa$ is small, all CRN-based methods converge successfully, while ACQRN achieves lower iteration counts and shorter CPU times. Compared with D-CRN, ACQRN avoids the computational overhead caused by the Nesterov-acceleration inner loop. Compared with V-CRN, ACQRN avoids the explicit computation of third-order derivatives. As the problem becomes increasingly ill-conditioned, the performance of D-CRN deteriorates rapidly. This behavior is consistent with its theoretical design: to guarantee descent, the cubic regularization parameter must scale with the Lipschitz constant of $\nabla^2\Phi(x)$, which is of order $\mathcal O(\kappa^3)$. Consequently, D-CRN is forced to take increasingly conservative steps. By contrast, the cubic regularization parameter in ACQRN scales only as $\mathcal O(\kappa)$, making it substantially more robust to ill-conditioning. Moreover, ACQRN outperforms the efficient GDA-BB method in the high-precision and ill-conditioned regimes. To further isolate the effect of ill-conditioning, Figure~\ref{fig:kappa-runtime} reports the CPU time required by each deterministic method over the finer grid $\kappa\in\{3,10,20,\ldots,100\}$; for methods that do not reach the target tolerance within the prescribed iteration budget, we report their final CPU time. Our ACQRN remains nearly insensitive to the growth of $\kappa$. 

\begin{figure}[H]
    \centering
    \includegraphics[width=\linewidth]{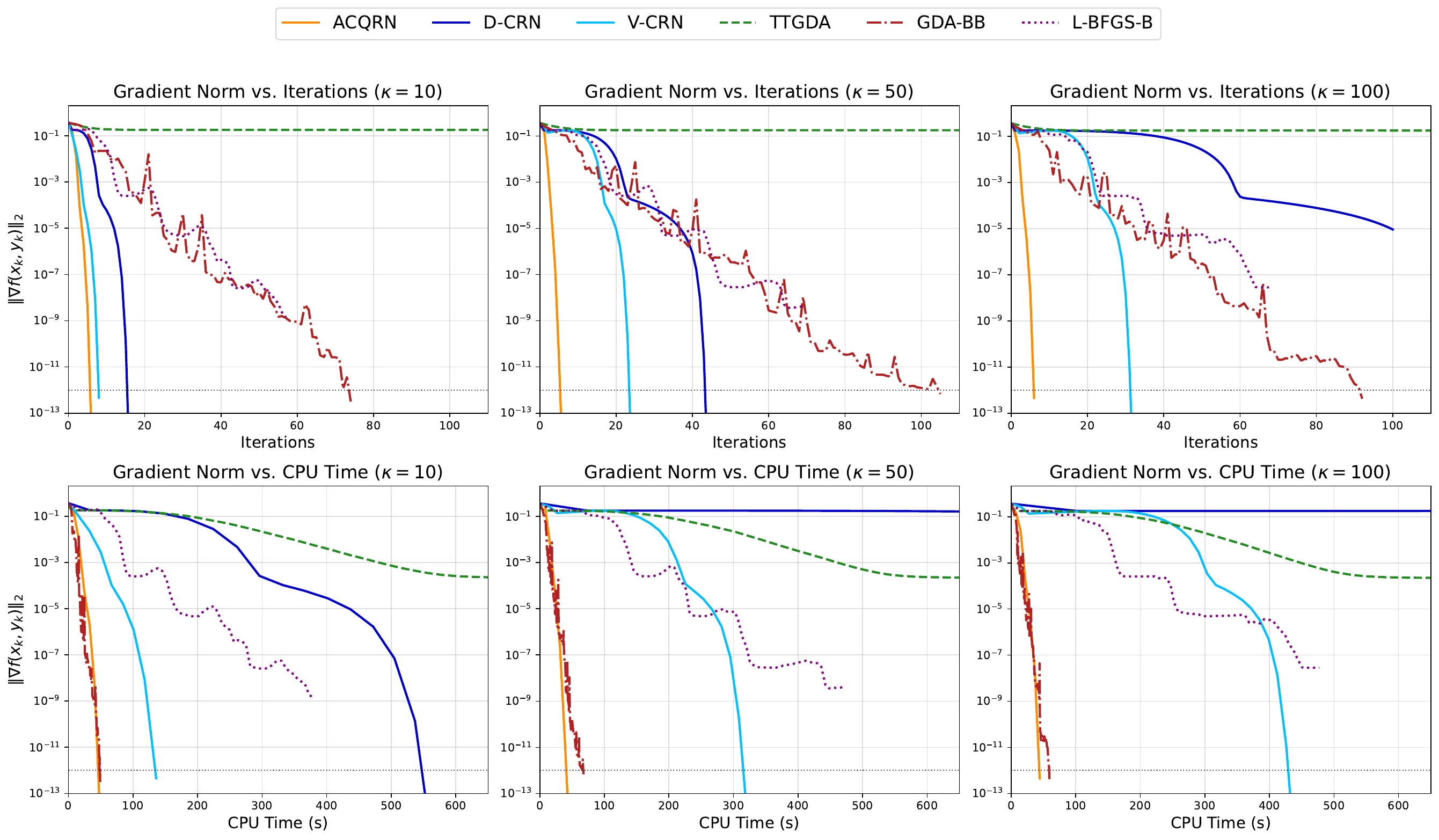}
    \caption{Deterministic robust-regression experiments for $\kappa\in\{10,50,100\}$. The top row reports $\|\nabla f(x_k,y_k)\|$ versus iteration, and the bottom row reports $\|\nabla f(x_k,y_k)\|$ versus CPU time. The horizontal dotted line represents the termination criterion $\|\nabla f\|\le 10^{-12}$.}
    \label{fig:cubic-deter}
\end{figure}

\vspace{-0.4cm}

\begin{figure}[H]
    \centering
    \includegraphics[width=0.6\linewidth]{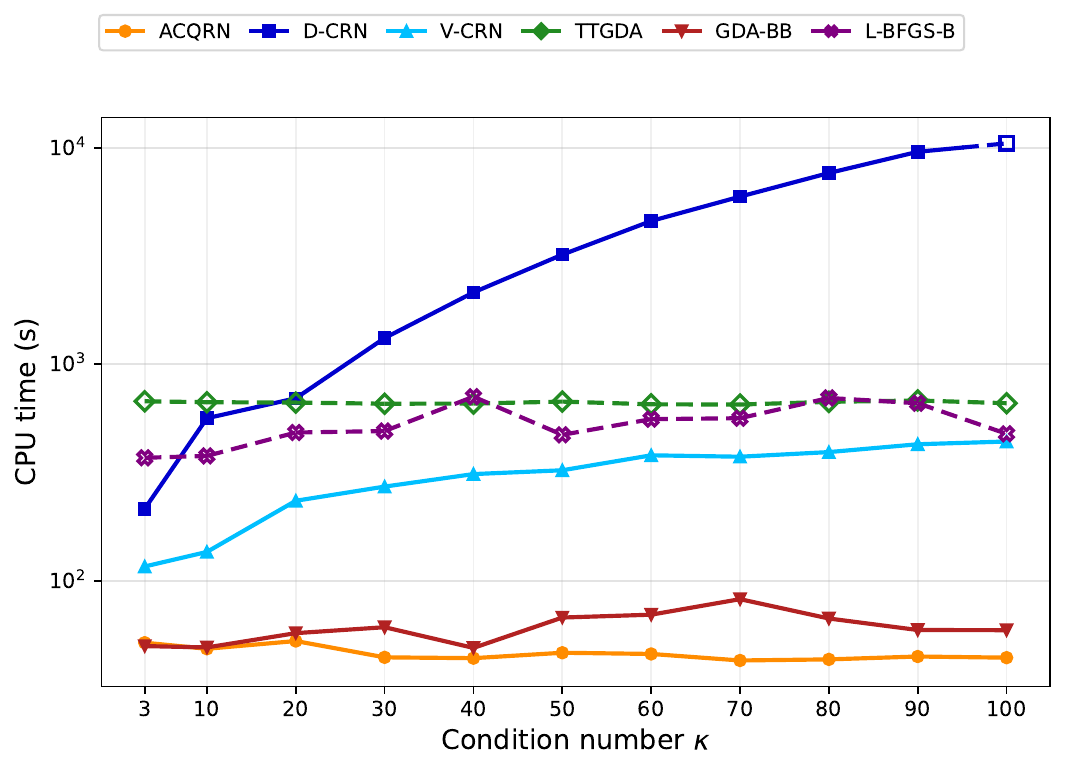}
    \caption{CPU time versus condition number in the deterministic robust-regression experiments over $\kappa\in\{3,10,20,\ldots,100\}$. Dashed line segments and hollow markers indicate runs that do not reach the target tolerance within the prescribed iteration budget; for these runs, the plotted values are final CPU times.}
    \label{fig:kappa-runtime}
\end{figure}

\subsection{Stochastic Setting}
To simulate the stochastic expectation setting, we sample from the dataset with replacement. Similar to the deterministic setting, we fix $\rho_x = 1$, initialize at the origin, and vary $\rho_y$ to generate instances with different condition numbers $\kappa$. For the tested stochastic CRN methods, the termination criterion is $|\Phi(x) - \Phi^*| \le 10^{-4}$ with a maximum of 100 outer iterations. For the tested stochastic first-order methods, the maximum iteration count is extended to 500. In the following we present the details of all tested algorithms.

\noindent \textbf{S-ACQRN}: our single-loop  S-ACQRN method (Algorithm \ref{alg:stoch-cubic-minimax}). 
We set $\beta = 2\mu^{-1}$, $\alo = 2\beta\rho$, and $\alt = 2(3\beta L + 1)\rho$, which match our theoretical analysis. Moreover, we implement an increasing batch size schedule which scales linearly with the epoch index. Specifically, we set the epoch lengths to $S^g = S^h = S = 5$. Let $E(k) = \lfloor k / S \rfloor + 1$ denote the current epoch index. For our Algorithm~\ref{alg:stoch-cubic-minimax}, the empirical batch sizes at iteration $k$ are
\begin{equation*}
\begin{aligned}
    |\cBgko| = |\cBgkt| &= \begin{cases}
     500 \cdot E(k), & \text{if } \operatorname{mod}(k, S) = 0;\\[1ex]
     100 \cdot E(k), &\text{otherwise};
    \end{cases}\\[2ex]
    |\cBhko| = |\cBhkt| &= |\cBgko| / 10, \qaq
    |\cBgky| = 100 \cdot E(k).
\end{aligned}
\end{equation*}
\noindent \textbf{D-SCRN}: the double-loop stochastic CRN method \citep[Algorithm 5]{chen2023a}. All hyperparameters, except for the batch sizes, are selected according to the theoretical results in \citep[Theorem 4]{chen2023a}\footnote{Because the theoretical bounds in \citep[Theorem 4]{chen2023a} often require the inner SGD loop to run for $10^{10}$ to $10^{13}$ iterations under our setting, we impose a practical upper limit of $10^4$ inner iterations to ensure the algorithm terminates within a reasonable amount of time.}. Let $\tilde{\mathcal{B}}^g_k$ and $\tilde{\mathcal{B}}^h_k$ denote D-SCRN's gradient and Hessian batch sizes at iteration $k$. We let its batch-size schedule follow a steadily increasing sequence proportional to the epoch indices:
\begin{equation}
\label{eq:D-SCRN-batch}
\begin{aligned}
    |\tilde{\mathcal{B}}^g_{k}| = 150 \cdot E(k) \qaq
    |\tilde{\mathcal{B}}^h_{k}| = 15 \cdot E(k).
\end{aligned}
\end{equation}
This calibration ensures that S-ACQRN and D-SCRN have comparable Hessian-sample budgets in the subproblems over each epoch. \vspace{0.1cm}

\noindent \textbf{V-SCRN}: vanilla stochastic CRN method \citep{tripuraneni2018stochastic} applied to the minimization problem \eqref{Prob_Min} with $\beta = 2/\mu$. The batch sizes match \eqref{eq:D-SCRN-batch} and  stochastic estimates of $\nabla^2 \hb$ are constructed using independent sampling. \vspace{0.1cm}

\noindent \textbf{TTSGDA}: stochastic version of TTGDA method \citep{lin2020gradient,lin2025two}. The step sizes $\eta_y$ and $\eta_x$ are chosen from the combinations listed for TTGDA method. The gradient batch size matches \eqref{eq:D-SCRN-batch}. \vspace{0.1cm}
    
\noindent \textbf{VRSGDA}: a double-loop variance-reduced stochastic GDA method \cite[Algorithm 5]{luo2020stochastic}. The gradient batch size matches \eqref{eq:D-SCRN-batch} and all the other hyperparameters are selected according to \cite[Theorem 1]{luo2020stochastic}\footnote{\cite[Theorem 1]{luo2020stochastic} requires the inner loop to run for $\lceil1024 \times \kappa\rceil$ iterations per outer iteration. We impose a practical upper limit of 2000 inner iterations to ensure the algorithm terminates within a reasonable amount of time.}. \vspace{0.1cm}

\begin{figure}[htbp]
    \centering
    \includegraphics[width=\linewidth]{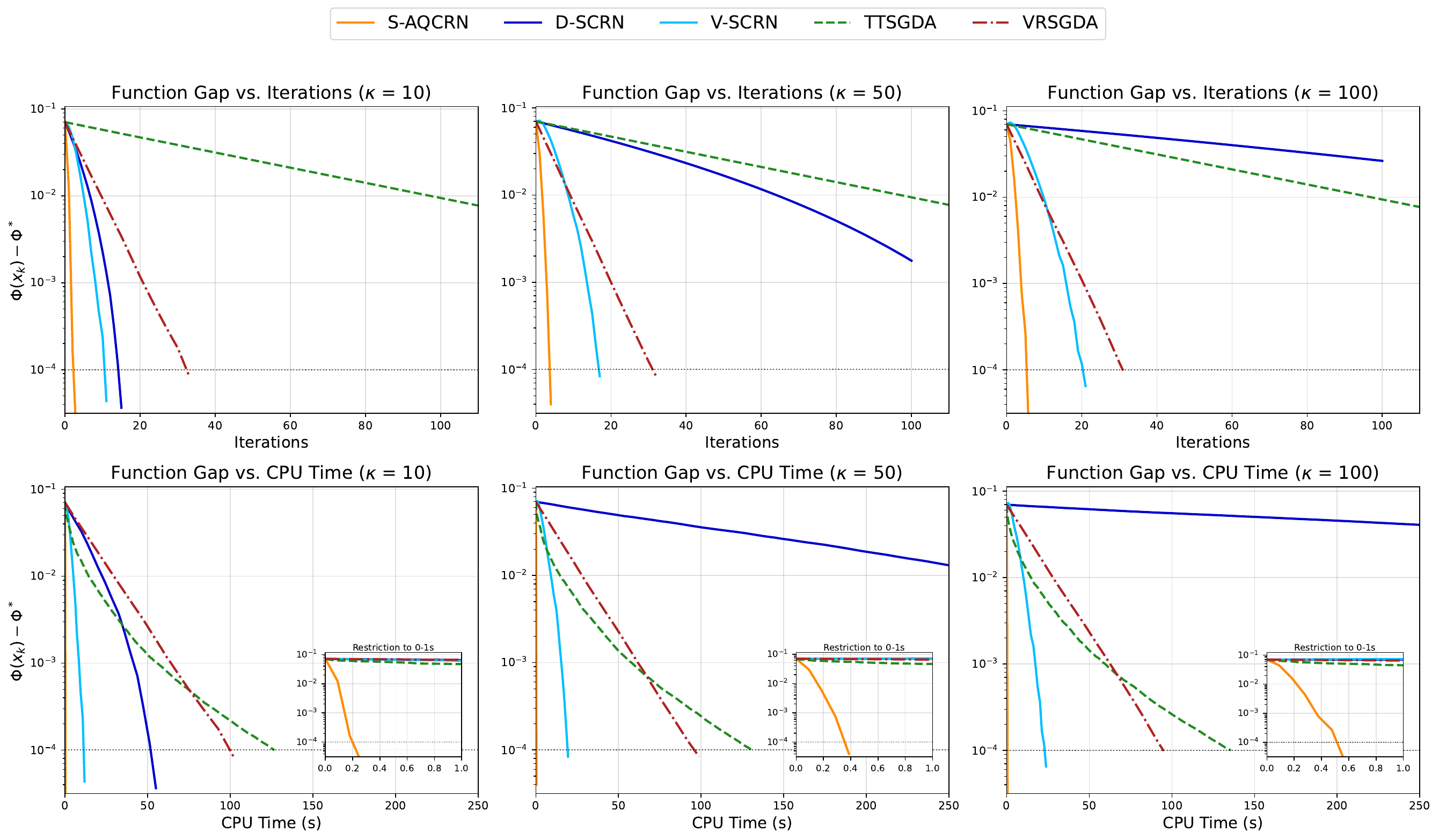}
    \caption{Stochastic robust-regression results for $\kappa\in\{10,50,100\}$. The top row reports $\Phi(x_k)-\Phi^*$ versus iteration, and the bottom row reports $\Phi(x_k)-\Phi^*$ versus CPU time. The horizontal dotted line represents the termination criterion $|\Phi(x) - \Phi^*|\le 10^{-4}$. The insets in the lower panels zoom in on the first second of CPU time.} 
    \label{fig:cubic-stoch}
\end{figure}

\begin{figure}[htbp]
    \centering
    \includegraphics[width=0.6\linewidth]{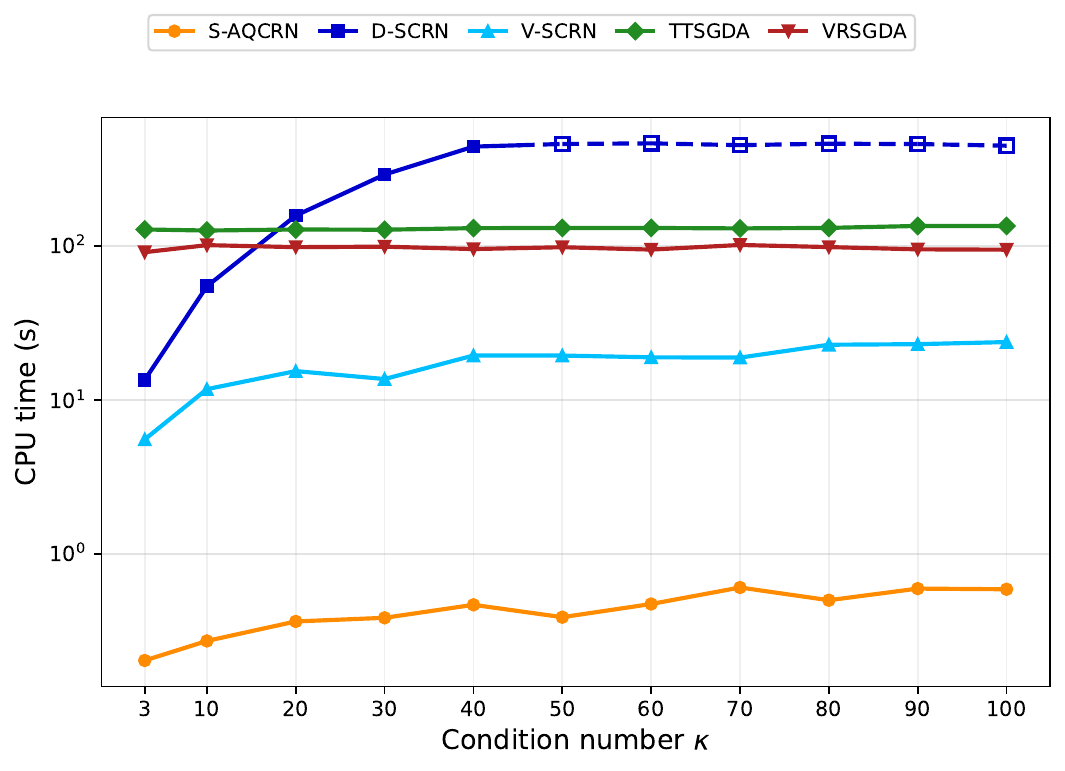}
    \caption{CPU time versus condition number in the stochastic robust-regression experiment over $\kappa\in\{3,10,20,\ldots,100\}$. Dashed line segments and hollow markers indicate runs that do not reach the target tolerance within the prescribed iteration budget; for these runs, the plotted values are final CPU times.}
    \label{fig:stoch-kappa-runtime}
\end{figure}

\noindent\textbf{Results for the stochastic setting.} 
The numerical results for the stochastic setting are reported in Figures~\ref{fig:cubic-stoch} and \ref{fig:stoch-kappa-runtime}. As in the deterministic setting, our S-ACQRN method has superior numerical performance compared to all the other tested methods. Furthermore, as $\kappa$ increases, the performance of the D-SCRN method is impaired by a cubic regularization parameter that scales as $\mathcal{O}(\kappa^3)$. The CPU runtime in Figure~\ref{fig:stoch-kappa-runtime} shows that S-ACQRN remains stable across the tested condition numbers, while D-SCRN reaches the iteration budget for the larger-$\kappa$ instances.

\section{Conclusion}
In this paper, we propose a novel single-loop cubic regularized Newton (CRN) framework for NC-SC minimax optimization. By minimizing a smooth regularized reformulation and introducing an Adaptive Cubic-Quadratic (ACQ) majorization scheme, our approach successfully bypasses the computational and theoretical bottlenecks inherent to existing double-loop methods. Consequently, we recover fast local superlinear convergence in the deterministic setting and achieve improved sample complexities in the stochastic setting. Furthermore, our in-expectation analysis resolves a critical technical gap in the existing stochastic nonconvex minimization literature. Looking forward, we emphasize that our regularized reformulation and the ACQ majorization technique are not strictly limited to cubic regularization; they provide a versatile and tractable foundation that can be naturally extended to design and analyze other single-loop second-order methods for the NC-SC minimax setting.






\vskip 0.2in
\bibliographystyle{plain}
\bibliography{references}
\end{document}